\documentclass[12pt]{article}

\usepackage{ucs}
\usepackage[utf8x]{inputenc}
\usepackage[UKenglish]{babel} 
\usepackage{babelbib} 
\usepackage{amsmath} 
\usepackage{amsthm} 
\usepackage{amssymb}
\usepackage{graphicx}
\usepackage{pdfpages}
\usepackage{verbatim}
\usepackage{color}
\usepackage{paralist}
\usepackage{bbm}
\usepackage{url}
\usepackage{hyperref}
\usepackage{float}
\usepackage{mathtools}
\usepackage{amsthm}
\usepackage{datetime}
\usepackage{setspace}
\usepackage{booktabs}
\usepackage{multirow}
\usepackage[a4paper, left=2.5cm, right=2.5cm, top=2.5cm,bottom=2.5cm]{geometry}

\newtheorem{theorem}{Theorem}
\newtheorem{lemma}{Lemma}
\newtheorem{definition}{Definition}
\newtheorem{proposition}{Proposition}

\newtheorem{remark}{Remark}

\newcommand{\E}{\mathbb{E}}

\newcommand{\R}{\mathbb{R}}

\title{Time-inconsistent view on a dividend problem with penalty}
\date{2022\\ April}
\author{Josef Anton Strini\thanks{Graz University of Technology, Institute of Statistics, Kopernikusgasse 24/III, 8010 Graz, Austria} \thanks{Contact: j.strini@tugraz.at} \and Stefan Thonhauser\footnotemark[1]  \thanks{Stefan Thonhauser received support by the Austrian Science Fund (FWF) Project P 33317.}}

\begin{document}

\maketitle

\begin{abstract}
We consider the dividend maximization problem including a ruin penalty in a diffusion environment. The additional penalty term is motivated by a constraint on dividend strategies. Intentionally, we use different discount rates for the dividends and the penalty, which causes time-inconsistency. This allows to study different types of constraints. For the diffusion approximation of the classical surplus process we derive an explicit equilibrium dividend strategy and the associated value function.
Inspired by duality arguments, we can identify a particular equilibrium strategy such that for a given initial surplus the imposed constraint is fulfilled. Furthermore, we illustrate our findings with a numerical example.\\

\end{abstract}

\section{Introduction}
The dividend maximization problem and many extensions of it are widely investigated in actuarial science and in particular in its risk theoretic branch. Under suitable conditions one can detect the optimal control in prespecified models. 
For example, in classical diffusion settings optimal dividend strategies are of constant barrier type, cf.~\cite{Shre} or \cite{AsmussenTaksar1997}. If using such a reflecting barrier strategy, this has the consequence of almost sure ruin.\\
On the other hand, one way to totally avoid ruin is the usage of capital injections, cf.~\cite{KuSchm}, one of the most renowned recent extensions of the problem. While these results are theoretically remarkable, they seem not to be satisfying if one aims at a balanced situation between risk and profit.
So there exist two extremal cases: on the one hand maximizing the output regardless of consequences - always having in mind that this will cause negative social impacts on the insured, employees, the public reputation of the insurance company and its management. On the other hand, keeping a business uncompromisingly alive using capital injections, may be in conflict with owner interest.\\
For this reason, we want to identify a dividend strategy which suffices a particular tradeoff between ruin and profit.
Naturally, one way of tackling this problem is to consider the dividend problem under a ruin constraint. That is maximizing expected discounted dividends over a set of strategies for which the ruin probability for fixed initial surplus stays below a given level. This approach goes back to \cite{Hipp2003}. Recent results in this direction are presented in \cite{Hipp2018a, Hipp2018b,Hipp2020} and mentioned in a talk by Brandon Israel Garc\'{\i}a Flores\footnote{\emph{Genetic Algorithms in Risk Theory: Optimizing Dividend Strategies} at {IME} conference 2021}. On a finite time horizon such a constrained dividend maximization problem is analyzed in \cite{Grandits2015}.\\
For dealing with the ruin probability constraint in the first place and for achieving a better understanding of the evolution of such constraints in the second place, we present a time-inconsistent variation or extension of \cite{Hernandez2015} in the following.\\
There are different origins of time-inconsistency in the dividend maximization problem. For example, \cite{Zhao2014} deal with the problem using non-exponential discounting, whereas \cite{ChrisLind2021} arrive at time(space)-inconsistency when adding a moment constraint on the number of expected dividend payments.\\

\subsection{Specification of the model}

In the following we will consider the surplus process $X=(X_t)_{t\geq 0}$ with the  dynamics
\begin{equation}\label{SDE}
dX_t = \mu (t,X_t) dt + \sigma(t,X_t) dW_t,
\end{equation}
where $\mu$ and $\sigma$ are deterministic functions such that \eqref{SDE} is well-posed and $W=(W_t)_{t\geq 0}$ is a standard Brownian motion on a probability space $(\Omega,\mathcal{F},\mathcal{P})$. The flow of information is given by the $\mathcal{P}$-completed Brownian history, denoted by $\{\mathcal{F}_t\}_{t\geq 0}$. 
Exemplary conditions which ensure that \eqref{SDE} admits a strong solution are the classical linear growth and Lipschitz conditions for $\mu$ and $\sigma$, see \cite[p.~289, Theorem 2.9]{KaSh}. Or, the more general requirements from \cite[p.~142, Theorem 4]{Zvonkin1974}, which are:
\begin{itemize}
\item $\mu$ is bounded and measurable,
\item  $\sigma$ is bounded from above and away from zero, is continuous in $(t,x)$ and is H{\"o}lder continuous with parameter larger or equal than $1/2$ with respect to $x$ for every $t$.
\end{itemize}
We subsequently refer to these latter conditions as conditions $(Z)$.\\
The task in the classical problem, cf.~\cite{Shre}, is to choose the dividend process  $L=(L_t)_{t\geq 0}$ in an optimal way, such that the discounted expected future dividends are maximized. The resulting value function is
\begin{equation}\label{valuefunction}
\sup_{L} \E_{x}\left[\int\limits_0^{\tau^L} e^{-\delta s} dL_s \right],
\end{equation}
where the supremum is taken over all processes $L$ which are non-decreasing, adapted, càglàd with $L_0=0$ and $\Delta L_t \leq X^L_{t-}$. The controlled surplus is
\begin{equation}\label{SDE controlled}
dX^L_t = \mu (t,X^L_t) dt + \sigma(t,X^L_t) dW_t- dL_t.
\end{equation}
For assumptions on $\mu$ and $\sigma$ such that \eqref{valuefunction} is well-posed, we  refer to \cite{Shre}.\\

In order to penalize early ruin \cite{ThonAlb2007} added the term $\Lambda e^{-\delta \tau^L}$ to the reward function. Based on their work, \cite{Hernandez2015} linked the penalty term to a constraint on $\E[e^{-\delta \tau^L}]$ in order to solve the associated constrained problem for a given initial surplus. Motivated by these previous results and by a comment in the recent contribution \cite[p.~636, Remark 2.1]{Juncaetal2019} on an extension with different discount rates, we want to analyze the following reward function for a $\lambda\leq 0$
$$\E_{t,x}\left[\int\limits_t^{\tau^L} e^{-\delta (s-t)} dL_s \right] + \lambda\left(  \E_{t,x}\left[e^{-\beta(\tau^L-t)} \right] - \alpha \right).$$
This integrates the constraint
\begin{equation}\label{constraint}
\E_{t,x}\left[e^{-\beta(\tau^L-t)} \right] \leq \alpha.
\end{equation}
Here $\beta$ and $\delta$ are positive, not necessarily equal constants, $0<\alpha \leq 1$ and $\tau^L=\tau^L(t,x)=\inf\{s\geq t| X^L_s \leq 0\,, X^L_t=x\}$, for a process $X^L$ starting in $x$ at time $t\geq 0$.\\
We consider the following variant of this problem. We assume that $dL_t = l(t,X^{L}_t) dt$ and we define the set of admissible controls $\Theta$ to be the measurable controls of feedback type $l:[0,\infty)^2 \to [0,L_{max}]$. The bound $L_{max}>0 $ is a prespecified constant.
The main purpose of considering controls of feedback type is due to the fact that we are facing a time-inconsistent stochastic control problem in the framework of \cite[p.~334]{Bjoerk2017} and \cite{ChrisLind2021}. 
Note that as in the classical time-consistent case, one could try to approach the singular control problem, where the dividend rate is unbounded enabling lump sum payments, by considering the limit of $L_{max}$ to infinity in the problem with restricted dividend payments, see \cite[p.~102]{schm}. But in the present case the underlying situation is different and the concept of equilibrium controls probably has to be adapted for the singular case. In order to keep our discussion as condensed as possible, we leave this task for future research.\\
\begin{remark}
Notice that our approach is also able to cover the situation with a constraint on the ruin probability. For this reason, we consider for the above expectation \eqref{constraint} the limit $\beta \to 0$. We have that $\E_{t,x}\left[e^{-\beta(\tau^L-t)} \right]=\E_{t,x}\left[e^{-\beta(\tau^L-t)} \left(1_{\{\tau^L <\infty \}} +1_{\{\tau^L =\infty \}}\right) \right]=\E_{t,x}\left[e^{-\beta(\tau^L-t)} 1_{\{\tau^L <\infty \}} \right]$ and $\lim\limits_{ \beta \to 0}\E_{t,x}\left[e^{-\beta(\tau^L-t)} 1_{\{\tau^L <\infty \}} \right] = P_{t,x}(\tau^L < \infty).$
\end{remark}

\begin{remark}
If we set $t=0$ and send $\beta$ to zero in \eqref{constraint}, we are able to identify the minimal argument $x$ for which a ruin probability constraint can be fulfilled for an $\alpha<1$ and constant $\mu$ and $\sigma$. This minimal $x$ appears, since the control process can only reduce the current reserve. Hence, if the constraint cannot be reached from the level $x$ without dividends, it is not possible to satisfy \eqref{constraint} with a control action. In order to determine this minimal level, we consider \eqref{SDE controlled} with $dL_t= l dt$, where $l$ is a positive constant. Consequently, we consider the hitting time of zero of a Brownian motion with drift $\mu - l>0$. The corresponding probability of ruin for this process is then given by
\begin{equation}\label{ruinprob}
\psi(x,l)= e^{\frac{-2(\mu-l)}{\sigma^2}x},
\end{equation}
see \cite[p.~109, Lemma 5.21]{schmiRi}.
Certainly, if the drift $\mu- l$ is negative, we cannot fulfil the constraint \eqref{constraint}.
In the extremal case $l=0$, we can satisfy \eqref{constraint} only if
\begin{equation*}
x \geq  - \frac{\sigma^2}{2} \frac{\ln(\alpha)}{\mu} > 0.
\end{equation*}
\end{remark}

\subsection{Reformulation of the reward}
Optimization problems with constraints can usually be handled using the Lagrange multiplier approach. For $\lambda \leq 0$ we define the reward function to be:
\begin{subequations}
\begin{equation}\label{objective}
J(t,x,l)=  \E_{t,x}\left[\int\limits_t^{\tau^{l} } e^{-\delta(s-t)} l(s,X^{l}_s) ds + \lambda \left(e^{-\beta(\tau^{l} -t)}  - \alpha\right) \right],
\end{equation}
and for $\beta \to 0$
\begin{equation}\label{objective0}
J_0(t,x,l)=  \E_{t,x}\left[\int\limits_t^{\tau^{l} } e^{-\delta(s-t)} l(s,X^{l}_s) ds + \lambda \left( 1_{\{\tau^{l} <\infty \}} - \alpha\right) \right],
\end{equation}
\end{subequations}
subject to
\begin{equation}\label{stateprocess}
dX^{l}_t = (\mu(t,X^{l}_t) - l(t,X^{l}_t)) dt + \sigma(t,X^{l}_t) dW_t.
\end{equation}
Moreover, let $\mu(t,x) - l(t,x)$ and $\sigma(t,x)$ be such that a solution to \eqref{stateprocess} exists. In our case we assume that $l(\cdot,\cdot) \in \Theta$ and that $\mu(\cdot,\cdot)$ and $\sigma(\cdot,\cdot)$ fulfil the previously stated conditions $(Z)$. Further, we obtain that the solution to \eqref{stateprocess} is a strong Markov process with infinitesimal generator
$$ \mathcal{A}^l g(t,x)= \frac{\partial}{\partial t} g(t,x)+ (\mu(t,x)-l(t,x)) \frac{\partial}{\partial x} g(t,x) +  \frac{\sigma^2(t,x)}{2} \frac{\partial^2}{\partial x^2}g(t,x),$$
for a suitable function $g$ and $(t,x) \in \R^+ \times \R$, see \cite[p.~322]{KaSh}.

\begin{remark}\label{remark0}
According to \cite[p.~137, eq.~(3)]{Hernandez2015} and \newline \cite{Hernandez2018} we set $\lambda \leq 0$ in the function \eqref{objective}, such that, if the constraint is fulfilled, the term on the right hand side is positive and corresponds to a penalization. We restrict ourselves to a fixed value $\lambda \leq 0$. As proposed in \cite[p.~259-260]{Okse}, we can first solve the problem for every fixed $\lambda$, and then choose $\lambda$ such that the constraint \eqref{constraint} is fulfilled. This is treated in detail in Section \ref{sec:constraint}. Nevertheless, the duality approach, especially as used in \cite[p.~139, Th.~4.2]{Hernandez2015}, is not completely applicable in our setting. This is because the game-theoretic approach, which we are going to use, is different to a classical maximization of a reward function.
\end{remark} 
Since $\delta \neq \beta$, the objective \eqref{objective} (or \eqref{objective0}) is time-inconsistent and the dynamic programming principle is not valid anymore.
Therefore, we make use of the theory of time-inconsistent stochastic optimal control, see \cite{EkelandPirvu2008}, \cite{EkelandMbodjiPirvu2012} and \cite{Bjoerk2017}.
For this purpose, we modify the reward function from \eqref{objective} and introduce a new variable $r \in [0,\infty)$, some artificial reference point in time. For $\beta>0$, we define the function $H: \R_0^+ \times \R_0^+ \times [0,L_{max}] \to \R$, to be
$$ H\left(r,s,l\right)= e^{-\delta(s-r)}l-\lambda \beta e^{-\beta(s-r)},$$
such that we arrive at a representation only using a running reward and a constant term. We denote this auxiliary reward function by:
\begin{subequations}
\begin{align}\label{auxiliaryrewardfunction}
\begin{split}
J(r,t,x,l)&=  \E_{t,x}\left[\int\limits_t^{\tau^{l}} e^{-\delta(s-r)}l(s,X_s^{l})-\lambda \beta e^{-\beta(s-r)} ds +  \lambda \left(1- \alpha\right) \right]\\
&=\E_{t,x}\left[\int\limits_t^{\tau^{l}} H\left(r,s,l(s,X_s^{l})\right)ds  +  \lambda \left(1- \alpha\right) \right].
\end{split}
\end{align}
For $\beta \to 0$ we have:
\begin{align}\label{auxiliaryrewardfunction0}
\begin{split}
J_0(r,t,x,l)&=  \E_{t,x}\left[\int\limits_t^{\tau^{l} } e^{-\delta(s-r)} l(s,X_s) ds + \lambda \left(1_{\{\tau^{l} <\infty \}}  - \alpha\right) \right].
\end{split}
\end{align}

\end{subequations}
For $r=t$ we obtain that the auxiliary reward function coincides with the original reward function $J(t,t,x,l)=J(t,x,l)$ and $J_0(t,t,x,l)=J_0(t,x,l)$, respectively.

\begin{remark}\label{remark1}
Note that we have the following integrability property of the reward. For all fixed $r\in \R_0^+$ and for all $(t,x) \in  \R_0^+ \times \R$ and each admissible control $l$, it holds that
\begin{equation}\label{intcondH}
E_{t,x}\left[\int\limits_t^{\tau^{l}} \left|H\left(r,s,l(s,X_s^{l})\right)\right|ds  \right]<\infty.
\end{equation}
Using the special form of $H$ and the fact that controls are bounded by $L_{max}>0$, we obtain that the auxiliary reward function and the original reward function are both bounded in $t$ and $x$. Namely, it holds that
\begin{equation*}
 \E_{t,x}\left[\int\limits_t^{\tau^{l}}\left| H\left(r,s,l(s,X_s^{l})\right)\right|ds \right] \leq\\
 \frac{L_{max}}{\delta} e^{-\delta (t-r)}+ |\lambda|e^{-\beta(t-r)}\leq \frac{L_{max}}{\delta} e^{\delta r}+ |\lambda|e^{\beta r}.
\end{equation*}
Moreover, in the special case $r=t$ we obtain that
\begin{equation*}
 \E_{t,x}\left[\int\limits_t^{\tau^{l}}\left| H\left(t,s,l(s,X_s^{l})\right)\right|ds \right] \leq \frac{L_{max}}{\delta}+ |\lambda|.
\end{equation*}
\end{remark}

\section{Equilibrium control and associated equilibrium value function}
The time-inconsistencies in \eqref{auxiliaryrewardfunction} and \eqref{auxiliaryrewardfunction0} prevent us from using the dynamic programming approach for deriving a maximizing dividend strategy. Instead we rely on the methodology introduced by \cite{Bjoerk2017}. 
More precisely, we are guided by \cite[p.~345]{Bjoerk2017}, \cite[p.~23]{Yong2012} and \cite[p.~3]{Zhao2014}. Analogously, as discussed in \cite{Bjoerk2017} the subsequent definition is motivated by an intrapersonal game of the management in force with its future composition. In our problem the executive management has to choose the dividend strategy, but their associated reward function is of the form \eqref{auxiliaryrewardfunction} or \eqref{auxiliaryrewardfunction0}, so their preferences are time-inconsistent and hence vary over time. The management of the company is naturally not appointed for life, hence it will happen that single executives or even the whole leadership changes over time. Of course, they should always act in favour of the shareholders of the company. Aiming at a time-consistent management decision results in various economic benefits. Such a line of action will improve both the stability inside the company, regarding affiliates and employees, and the general external assessment, affecting future investors and clients. In the light of these economic considerations, the present constraint \eqref{constraint}, discussed in Section \ref{sec:constraint}, also can be viewed as a comprehensive alternative to classical one-year risk considerations, e.g. conditions on the value at risk.

\begin{definition}[{\cite[Def.~3.4, p.~336]{Bjoerk2017}}]\label{eqlaw}
Consider an admissible control $\hat{l} \in \Theta$ and choose an other arbitrary admissible control $l \in \Theta$, a fixed value $h >0$ and an arbitrary but fixed initial point $(t,x)$. Consider the control $l_h$ defined by
\begin{equation}\label{perturbedcontrol}
l_h(s,y) = \begin{cases}
l(s,y), \text{ for } \ t\leq s < t+h, \ y \in \R^+,\\
\hat{l}(s,y), \text{ for } \ s\geq t+h, \ y \in \R^+.
\end{cases}
\end{equation}
Now, if for all $l \in \Theta$ $$\liminf\limits_{h\to 0}\frac{J(t,x,\hat{l})- J(t,x,l_h)}{h} \geq 0,$$
$\hat{l}$ is called equilibrium control and the associated (equilibrium) value function is defined via 
\begin{equation}\label{equilibriumvaluefunction}
V(t,x)\coloneqq J(t,x,\hat{l}).
\end{equation}
\end{definition}

\begin{definition}\label{HJB}
The extended Hamilton-Jacobi-Bellman system of equations for a function $g(r,t,x)=g^r(t,x)$ and the given objective \eqref{objective} (and for $\beta \to 0$ \eqref{objective0}) reads as follows for all $t, x \in [0,\infty)$ 
\begin{subequations}
\begin{align}
\sup\limits_{l \in [0,L_{max}]} \{\mathcal{A}^{l}g^t(t,x) + H\left(t,t,l\right)\}=0,\label{HJB1}
\end{align}
\begin{align}
\sup\limits_{l \in [0,L_{max}]} \{\mathcal{A}^{l}g^t(t,x) + l\}=0,\label{HJB1_0}
\end{align}
\end{subequations}
 and for all $r,t,x \in [0,\infty)$
 \begin{subequations}
\begin{align}
\mathcal{A}^{\hat{l}}g^r(t,x) + H\left(r,t,\hat{l}\right)=0\label{HJB2}
\end{align}
\begin{align}
\mathcal{A}^{\hat{l}}g^r(t,x) + e^{-\delta(t-r)}\hat{l}=0\label{HJB2_0}
\end{align}
\end{subequations}
\begin{subequations}
\begin{align} 
\lim\limits_{w\to \infty} \E_{t,x}\left[ g^r(w\wedge \tau^{\hat{l}}, X^{\hat{l}}_{w\wedge \tau^{\hat{l}}})\right]=\lambda \left(1- \alpha\right).\label{HJB3}
\end{align}
 \begin{align} 
\lim\limits_{w\to \infty} \E_{t,x}\left[ g^r(w\wedge \tau^{\hat{l}}, X^{\hat{l}}_{w\wedge \tau^{\hat{l}}})\right]=\lambda \left( \E_{t,x}\left[1_{\{\tau^{\hat{l}} <\infty \}}\right]- \alpha\right).\label{HJB3_0}
\end{align}
\end{subequations}
Moreover, $\hat{l}$ denotes the control for which the supremum is attained. Where \eqref{HJB1}, \eqref{HJB2} and \eqref{HJB3} correspond to the case $\beta>0$ and \eqref{HJB1_0}, \eqref{HJB2_0} and \eqref{HJB3_0} belong to the case $\beta \to 0$.
\end{definition}

\begin{remark}\label{growthcond}
Notice that the condition \eqref{HJB3} (or \eqref{HJB3_0}) is the natural one in order to obtain the representation \eqref{g} (or \eqref{g_0}). In this instance we are conscious that the representation \eqref{HJB3_0} is not exactly in line with \cite{Bjoerk2017}, since following their representation, we would have to consider the function $$h_0(t,x)=\lambda \left( \E_{t,x}\left[1_{\{\tau^{\hat{l}} <\infty \}}\right]- \alpha\right)$$ with the associated equation $\mathcal{A}^{\hat{l}}h_0(t,x)=0$ and the corresponding boundary conditions ($\lambda(1-\alpha)$ at $x=0$ and $-\lambda \alpha$ for $x$ to infinity). But since we want to use a verification argument to confirm our solution as the value function, we prefer boundary conditions in terms of deterministic functions. Hence, we impose the following conditions which imply \eqref{HJB3} (or \eqref{HJB3_0}):
\begin{align}\label{HJB4}
\begin{split}
g^r(t,0)&=\lambda(1-\alpha)\quad (r,t) \in [0,\infty) \times [0,\infty),\\
\lim\limits_{x\to \infty} g^r(t,x)&=e^{-\delta( t-r)} \frac{L_{max}}{\delta} - \lambda e^{-\beta(t-r)} + \lambda(1-\alpha) \quad (r,t) \in [0,\infty) \times [0,\infty),
\end{split}
\end{align}
in addition to the requirement that $g^r$ is bounded.
The above requirements \eqref{HJB4} for $\beta \to 0$ also yield \eqref{HJB3_0}, since
\begin{align*}
&\lim\limits_{w\to \infty} \E_{t,x}\left[ g^r(w\wedge \tau^{\hat{l}}, X^{\hat{l}}_{w\wedge \tau^{\hat{l}}})\right]=\\
 &\E_{t,x}\left[ 1_{\{\tau^{\hat{l}} =\infty \}}(-\lambda \alpha) \right] + \E_{t,x}\left[ 1_{\{\tau^{\hat{l}} <\infty \}}\lambda(1- \alpha) \right] =\lambda \left( \E_{t,x}\left[1_{\{\tau^{\hat{l}} <\infty \}}\right]- \alpha\right).
\end{align*}
\end{remark}

\section{Verification Theorem}

In this section we state and prove the associated \emph{verification theorem}.
The required regularity of the involved function is due to the applicability of a suitable version of It\^{o}'s formula and the properties of a constructible explicit solution.

\begin{theorem}[Verification Theorem]\label{verification}
Given a real-valued function $g \in \mathcal{C}^{0,1,1}\left([0,\infty)^3 \right)$ $\cap$ $\mathcal{C}^{0,1,2}\left([0,\infty) \times [0,\infty) \times ([0,\infty)\backslash \{b\}) \right)$ for a value $b\geq 0$, which is bounded for all $t \in[0,\infty)$ and for all $x \in [0,\infty)$. In addition, suppose that $\frac{\partial}{\partial x}g(r,t,x)$, $\frac{\partial^2}{\partial x^2}g(r,t,x)$ and $\frac{\partial}{\partial t}g(r,t,x)$ are bounded for all $t \in[0,\infty)$ and for all $x \in [0,\infty)$. Furthermore, let $g$ solve the extended HJB system in Definition \ref{HJB}, where $\hat{l}$ is an admissible control for which the supremum is attained.\\
Then,
\begin{itemize}
\item[-]$\hat{l}$ is an equilibrium control according to Definition \ref{eqlaw},
\item[-] $V(t,x)=g(t,t,x)$ is the associated value function and
\item[-] $g^r(t,x)=g(r,t,x)$ has the representation, for all $r\in [0,\infty)$:
\begin{subequations}
\begin{equation}\label{g}
g^r(t,x)=\E_{t,x}\left[ \int\limits_{t}^{\tau^{\hat{l}}}   e^{-\delta(s-r)}\hat{l}(s,X_s^{\hat{l}})- \lambda \beta e^{-\beta(s-r)}ds+\lambda \left(1- \alpha\right)\right] \text{ for } \beta>0,
\end{equation}
and
\begin{equation}\label{g_0}
g^r(t,x)=\E_{t,x}\left[ \int\limits_{t}^{\tau^{\hat{l}}}   e^{-\delta(s-r)}\hat{l}(s,X_s^{\hat{l}})ds+\lambda \left(1_{\{\tau^{\hat{l}} <\infty \}}- \alpha\right)\right]  \text{ for } \beta \to 0.
\end{equation}
\end{subequations}
\end{itemize}
\end{theorem}

\begin{proof}
At first, we are going to show that $g(t,t,x)=J(t,x,\hat{l})$. We are certainly guided by the work of \cite[p.~345]{Bjoerk2017}, \cite[p.~23]{Yong2012} and \cite[p.~3]{Zhao2014}.\\
We start with showing the stochastic representation of $g^r$, for that reason we use a modified version of It\^{o}'s formula, cf.~\cite[p.~7, Thm.~2.9]{LeobThonSz2015}, and take expectations. Let $r,t,x \in [0,\infty)$ and $w> t$ be fixed, further we set $v= w\wedge \tau^{\hat{l}}$ and use the bounded derivatives of $g$ in addition to \eqref{HJB2} to obtain
\begin{align*}
&\E_{t,x}\left[ g^r(v, X^{\hat{l}}_{v})- g^r(t,x)\right]=\\
&\E_{t,x}\left[ \int\limits_{t}^{v} \mathcal{A}^{\hat{l}} g^r(s,X^{\hat{l}}_s)ds\right] =\E_{t,x}\left[ \int\limits_{t}^{v}  -H\left(r,s,\hat{l}(s,X^{\hat{l}}_s)\right)ds\right].
\end{align*}
Now the usage of boundary conditions from \eqref{HJB3}, the boundedness as discussed in Remark \ref{remark1} and the dominated convergence theorem (or even monotone convergence if $\lambda\leq 0$ because $H\geq0$) leads to the desired form
\begin{align*}
g^r(t,x)=&\E_{t,x}\left[ \int\limits_{t}^{\tau^{\hat{l}}}  H\left(r,s,\hat{l}(s,X^{\hat{l}}_s)\right)ds+\lambda \left(1- \alpha\right)\right]\\
=&\E_{t,x}\left[ \int\limits_{t}^{\tau^{\hat{l}}}   e^{-\delta(s-r)}\hat{l}(s,X_s^{\hat{l}})- \lambda \beta e^{-\beta(s-r)}ds+\lambda \left(1- \alpha\right)\right].
\end{align*}
This gives indeed $$g^t(t,x)=g(t,t,x)=J(t,x,\hat{l}). $$
In the second part we prove that $\hat{l}$ is indeed an equilibrium control. Therefore, we consider, as proposed in Definition \ref{eqlaw}, an arbitrary admissible control $l \in \Theta$, $h>0$ and arbitrary but fixed $t,x \in [0,\infty)$ in order to define $l_h(s,y)$ for $s \in [t, \infty)$ and $y \in [0,\infty)$ as in \eqref{perturbedcontrol}. Set $v=(t+h) \wedge \tau^{l_h}(t,x)$, then we have
\begin{align*}
&J(t,x,\hat{l})- J(t,x,l_h)\\
&=g(t,t,x)-  \E_{t,x}\left[ \int\limits_{t}^{v}  H\left(t,s,l_h(s,X^{l_h}_s)\right)ds+\lambda \left(1- \alpha\right)\right]\\
 &- \E_{t,x}\left[ \E_{v,X^{l_h}_v}\left[\int\limits_{v}^{\tau^{l_h}(t,x)}  H\left(t,s,l_h(s,X^{l_h}_s)\right) ds\right] \right].
\end{align*}
The first subsequent equality holds due to the definition of $l_h$ and the strong Markov property. For the second one we need that \eqref{g} holds true for different values in the first and second argument of $g$ ,
\begin{align*}
&\E_{t,x}\left[ \E_{v,X^{l_h}_v}\left[\int\limits_{v}^{\tau^{l_h}(t,x)}  H\left(t,s,l_h(s,X^{l_h}_s)\right) ds + \lambda(1-\alpha)\right] \right]\\
&=\E_{t,x}\left[ \E_{v,X^{l_h}_v}\left[\int\limits_{v}^{\tau^{\hat{l}}(v,X^{l_h}_v)}  H\left(t,s,\hat{l}(s,X^{\hat{l}}_s)\right) ds+ \lambda(1-\alpha)\right] \right]\\
&=\E_{t,x}\left[ g(t,v,X^{l_h}_v)\right].
\end{align*}
Moreover, taking this into account and using the version of It\^{o}'s formula from \newline
\cite[p.~7, Thm.~2.9]{LeobThonSz2015}, we end up with
\begin{align}\label{proof1}
\begin{split}
&J(t,x,\hat{l})- J(t,x,l_h)\\
&=\E_{t,x}\left[ g(t,t,x)  - g(t,v,X^{l_h}_v) - \int\limits_{t}^{v}  H\left(t,s,l_h(s,X^{l_h}_s)\right)ds \right]\\
&= - \E_{t,x}\left[  \int\limits_{t}^{v} \mathcal{A}^{l_h} g^t(s,X^{l_h}_s) + H\left(t,s,l_h(s,X^{l_h}_s)\right)ds \right].
\end{split}
\end{align}

The expectation of the stochastic integral part of \eqref{proof1} is zero. This holds true since $\frac{\partial}{\partial x}g$  and $\sigma$ are  continuous in $t$ and $x$ and the controlled process is adapted and continuous, we get that the integrand of the stochastic integral is predictable. Furthermore, $\frac{\partial}{\partial x}g$  and $\sigma$ are bounded in $t$ and $x$ and hence the stochastic integral is indeed a martingale for general admissible controls $l$.\\
Further, we divide  \eqref{proof1} by $h$ and consider the limit inferior for  $h$ to zero. Using bounded convergence to interchange limit and expectation and Lebesgue's differentiation theorem \cite[p.~100, 108-109]{Wheeden1977}, we obtain the second equality in \eqref{proof2}
\begin{align}\label{proof2}
\begin{split}
&\liminf\limits_{h\to 0}\frac{J(t,x,\hat{l})- J(t,x,l_h)}{h}\\
 &=\liminf\limits_{h\to 0} \E_{t,x}\left[ -\frac{1}{h}\int\limits_{t}^{v} \mathcal{A}^{l_h} g^t(s,X^{l_h}_s) + H\left(t,s,l_h(s,X^{l_h}_s)\right)ds \right]\\
&= -  \left(\mathcal{A}^{l} g^t(t,x) + H\left(t,t,l(t,x)\right) \right)\geq 0.
\end{split}
\end{align}
Moreover, we use \eqref{HJB1} for the last inequality in \eqref{proof2} to finally obtain that $\hat{l}$ is an equilibrium control. This gives that $V(t,x)=J(t,x,\hat{l})=g(t,t,x)$.
\end{proof}
\begin{remark}
In the theory of stochastic optimal control, the classical solution approach suggests to prove a verification theorem and present a suitable solution satisfying the needed regularity requirements. For the sake of completeness, we want to mention the result by \cite{lindensjoe2019}, where it is shown that, under regularity assumptions, it is even necessary for the equilibrium to solve the extended HJB system. We adopt the notation of the following theorem and just mention that our defined extended HJB equation is a special case of the extended HJB system II in \cite{lindensjoe2019}. We omit further details and refer to the original source.

\begin{theorem}[ {\cite[p.~431, Theorem 3.12]{lindensjoe2019}} ]
A regular equilibrium $(\hat{u},V_{\hat{u}},f_{\hat{u}},g_{\hat{u}})$  solves the extended HJB system II.
\end{theorem}
\end{remark}

\section{Classical solution approach}
In this section we are able to reveal an explicit solution to the above stated time-inconsistent problem under the assumption of constant coefficients, i.e., $\mu(t,x)\equiv \mu$ and $\sigma(t,x) \equiv \sigma$.
\begin{remark}\label{betatozero}
Please notice that in the subsequent analysis we can send $\beta \to 0$ and do not violate its general validity if $\mu > L_{max}$. Since in this instance the function $V_2(x;b)$, which will be constructed subsequently, solves the associated ODE \eqref{PDE2} also in the limiting case.\\
On the contrary, if $\mu \leq L_{max}$ this function $V_2(x;b)$ is zero in the limiting case and cannot fulfil the condition $\lim\limits_{x\to \infty}  V_2(x;b)= -\lambda$ in \eqref{PDE2}. Therefore, the present problem for $\mu \leq L_{max}$ reduces, as $\beta \to 0$, to the unconstrained classical dividend problem with the value function shifted by the constant $\lambda(1-\alpha)$ and optimal threshold height $b^*=\bar{b}$.
\end{remark}
We follow the classical line of attack as in \cite{Zhao2014}. Since we have to solve the equation \eqref{HJB1},
\begin{align*}
&\sup\limits_{l \in [0,L_{max}]} \left\{\frac{\partial}{\partial t} g^t(t,x) + (\mu-l)\frac{\partial}{\partial x} g^t(t,x)  + \frac{\sigma^2}{2 } \frac{\partial^2}{\partial x^2} g^t(t,x) + l -\lambda \beta \right\}=\\
&\frac{\partial}{\partial t} g^t(t,x) + \mu \frac{\partial}{\partial x} g^t(t,x)  + \frac{\sigma^2}{2 } \frac{\partial^2}{\partial x^2} g^t(t,x) -\lambda \beta + \sup\limits_{l \in [0,L_{max}]} \left\{l\left(1-\frac{\partial}{\partial x} g^t(t,x)\right) \right\}=0,
\end{align*}
we obtain as in the classical problem for the maximizing argument
\begin{align*}
l(t,x)=\left\{\begin{array}{lr}
        0, & \text{if } \frac{\partial}{\partial x} g^t(t,x)>1,\\
        L_{max}, & \text{if } \frac{\partial}{\partial x} g^t(t,x)\leq1.
        \end{array}\right.
\end{align*}
Moreover, if we assume that our function is increasing and concave in $x$, there exists a threshold level $b$ such that $\frac{\partial}{\partial x} g(t,t,x)> 1$ for $0\leq x< b$ and $\frac{\partial}{\partial x} g(t,t,x)\leq  1$ for $x\geq b$. Under this premise and taking account of Remark \ref{growthcond}, fulfilling the following equation is sufficient for the solution of the extended HJB equation in Definition \ref{HJB}:
\begin{align}\label{HJBbarrier}
\left\{\begin{array}{ll}
0=\frac{\partial}{\partial t} g(r,t,x) +  \mu \frac{\partial}{\partial x} g(r,t,x)  + \frac{\sigma^2}{2 } \frac{\partial^2}{\partial x^2} g(r,t,x) -\lambda \beta e^{-\beta(t-r)},\\
\quad \text{ for }(r,t,x) \in [0,\infty) \times [0,\infty) \times [0,b),\\     
0=\frac{\partial}{\partial t} g(r,t,x) +  (\mu-L_{max})\frac{\partial}{\partial x} g(r,t,x)  + \frac{\sigma^2}{2 } \frac{\partial^2}{\partial x^2} g(r,t,x)+ e^{-\delta(t-r)}   L_{max}\\
-\lambda \beta e^{-\beta(t-r)},\\
\quad \text{ for }(r,t,x) \in [0,\infty) \times [0,\infty)\times [b,\infty),\\
g(r,t,0)=\lambda(1-\alpha),\\
\quad \text{ for } (r,t) \in [0,\infty) \times [0,\infty),\\
\lim\limits_{x\to \infty} g(r,t,x)=e^{-\delta( t-r)} \frac{L_{max}}{\delta} - \lambda e^{-\beta(t-r)} + \lambda(1-\alpha),\\
\quad \text{ for }(r,t) \in [0,\infty) \times [0,\infty).
\end{array}\right.
\end{align}
In order to take advantage of Theorem \ref{verification} we make the following well guessed ansatz:
\begin{equation}
\nu(r,t,x;b)=e^{-\delta(t-r)} V_1(x;b)+e^{-\beta(t-r)} V_2(x;b)+\lambda (1-\alpha),
\end{equation}
where $V_1(x;b)$ is the solution of 
\begin{align}\label{PDE1}
\left\{\begin{array}{lr}
-\delta v_1(x) + \mu \frac{\partial}{\partial x} v_1(x)  + \frac{\sigma^2}{2 } \frac{\partial^2}{\partial x^2} v_1(x) =0, &0\leq x< b\\     
 -\delta v_1(x) +(\mu-L_{max}) \frac{\partial}{\partial x} v_1(x)  + \frac{\sigma^2}{2 } \frac{\partial^2}{\partial x^2} v_1(x)+ L_{max} =0, &x\geq  b,\\
v_1(0)=0,\\
\lim\limits_{x\to \infty}  v_1(x)= \frac{L_{max}}{\delta},
        \end{array}\right.
\end{align}
and $V_2(x;b)$ solves
\begin{align}\label{PDE2}
\left\{\begin{array}{lr}
-\beta v_2(x) + \mu \frac{\partial}{\partial x} v_2(x)  + \frac{\sigma^2}{2 } \frac{\partial^2}{\partial x^2} v_2(x) -\lambda \beta =0, &0\leq x< b\\     
 -\beta v_2(x) +(\mu-L_{max}) \frac{\partial}{\partial x} v_2(x)  + \frac{\sigma^2}{2} \frac{\partial^2}{\partial x^2} v_2(x)-\lambda \beta =0, &x\geq b,\\
v_2(0)=0,\\
\lim\limits_{x\to \infty}  v_2(x)= -\lambda.
        \end{array}\right.
\end{align}
The above equations lead to the following solutions
\begin{align*}
V_1(x;b)=\left\{\begin{array}{lr}
        A_1(b) e^{a_1 x} + A_2(b) e^{a_2 x} , & \text{if } 0\leq x< b\\
        \frac{L_{max}}{\delta}+ B_1(b) e^{b_1 x} + B_2(b) e^{b_2 x}, & \text{if } x\geq b,
        \end{array}\right.
\end{align*}
where
\begin{align*}
a_{1,2}=-\frac{\mu}{\sigma^2} \pm \sqrt{\frac{\mu^2}{\sigma^4}+ \frac{2 \delta}{\sigma^2}}, \quad \quad 
b_{1,2}=-\frac{(\mu-L_{max})}{\sigma^2} \pm \sqrt{\frac{(\mu-L_{max})^2}{\sigma^4}+ \frac{2 \delta}{\sigma^2}}.
\end{align*}
Due to the conditions $v_1(0)=0$ and $ \lim\limits_{x\to \infty}  v_1(x)= L_{max} / \delta$, we obtain $A_2(b)=-A_1(b)$ and $B_1(b)=0$. On the other hand we have for $V_2$:
\begin{align*}
V_2(x;b)=\left\{\begin{array}{lr}
        -\lambda + C_1(b) e^{c_1 x} + C_2(b) e^{c_2 x} , & \text{if } 0\leq x< b\\
        -\lambda + D_1(b) e^{d_1 x} + D_2(b) e^{d_2 x}, & \text{if } x\geq b,
        \end{array}\right.
\end{align*}
where
\begin{align*}
c_{1,2}=-\frac{\mu}{\sigma^2} \pm \sqrt{\frac{\mu^2}{\sigma^4}+ \frac{2 \beta}{\sigma^2}}, \quad \quad
d_{1,2}=-\frac{(\mu-L_{max})}{\sigma^2} \pm \sqrt{\frac{(\mu-L_{max})^2}{\sigma^4}+ \frac{2 \beta}{\sigma^2}}.
\end{align*}
From the conditions $v_2(0)=0$ and $ \lim\limits_{x\to \infty}  v_2(x)= -\lambda $, we obtain $C_2(b)=\lambda-C_1(b)$ and $D_1(b)=0$.
Furthermore, we ask for the following smooth fit conditions for $i=1,2$, 
\begin{align}
&V_i(b-;b)=V_i(b+;b),\label{smoothfitcond1}\\ 
&V_i'(b-;b)=V_i'(b+;b) \text{ and } \label{smoothfitcond2}\\
&\frac{\partial}{\partial x} \nu(t,t,b-;b)=1.\label{smoothfitcond3}
\end{align}
Note that the condition \eqref{smoothfitcond3} is in contrast to the smooth fit condition of the classical dividend problem. In our case, we have to make sure that $V'_1(b-;b)+V'_2(b-;b)=1$.\\

For $V_1$ the first two conditions \eqref{smoothfitcond1} and \eqref{smoothfitcond2} immediately yield
\begin{align*}
A_1(b)&=- \frac{L_{max}}{\delta} \frac{b_2}{(a_1-b_2)e^{a_1 b}+ (b_2-a_2) e^{a_2 b}},\\
B_2(b)&= \frac{A_1(b)}{b_2 e^{b_2 b}} (a_1e^{a_1 b} - a_2 e^{a_2 b}).
\end{align*}
Since $b_2<0$, $a_1-b_2>0$ and $b_2-a_2>0$, we get that $A_1(b)>0$ and $B_2(b)<0$.
Analogously, \eqref{smoothfitcond1} and \eqref{smoothfitcond2} give for $V_2$ that 
\begin{align*}
C_1(b)&= \frac{\lambda(c_2-d_2) e^{c_2 b}}{(d_2-c_1)e^{c_1 b}+ (c_2-d_2) e^{c_2 b}},\\
D_2(b)&= \frac{C_1(b) e^{c_1 b} +(\lambda- C_1(b)) e^{c_2 b}}{e^{d_2 b}}.
\end{align*}
Because $c_2-d_2<0$ and $\lambda<0$, we obtain $\lambda-C_1(b)<0$,  $C_1(b)<0$ and $D_2(b)<0$.
The last condition \eqref{smoothfitcond3} reveals the exact candidate equilibrium threshold level $b$. Indeed, we are looking for a value $b$ such that
\begin{align*}
&A_1(b) (a_1 e^{a_1 b} - a_2 e^{a_2 b}) + C_1(b) c_1 e^{c_1 b} +(\lambda- C_1(b))c_2  e^{c_2 b}=\\
&\frac{L_{max}}{\delta} \frac{(-b_2)(a_1 e^{a_1 b} - a_2 e^{a_2 b})}{(a_1-b_2)e^{a_1 b}+ (b_2-a_2)e^{a_2 b}}  + \lambda \frac{(c_2-c_1) d_2 e^{(c_1+c_2)b}}{(d_2-c_1)e^{c_1b}+(c_2-d_2)e^{c_2 b}}=1.
\end{align*}
Therefore, we analyze the following function
\begin{align}\label{eq:G}
G(b)=\frac{L_{max}}{\delta} \frac{(-b_2)(a_1 e^{a_1 b} - a_2 e^{a_2 b})}{(a_1-b_2)e^{a_1 b}+ (b_2-a_2)e^{a_2 b}}  + \lambda \frac{(c_2-c_1) d_2 e^{(c_1+c_2)b}}{(d_2-c_1)e^{c_1b}+(c_2-d_2)e^{c_2 b}}-1.
\end{align}
We have to find a $b>0$ such that $G(b)=0$. If this root exists, we denote it by $b^*$. 
\begin{remark}\label{zeroequation}
Concerning the dependence of the equilibrium threshold level $b^*$ on $\beta$ we obtain that: if $\beta$ tends to zero, then $c_1=0$, $c_2=-(2\mu)/{\sigma^2}$, and if $\mu-L_{max}>0$, we obtain $d_1=0$ and $d_2=-2(\mu-L_{max})/{\sigma^2}<0$.\\
If $\mu-L_{max}\leq0$ we obtain $d_1=-2(\mu-L_{max})/{\sigma^2}\geq0$ and $d_2=0$. Which means that the function $G(b)$ reduces to
\begin{align}\label{eq:G0}
G_0(b)=\left\{\begin{array}{lr}
         \frac{L_{max}}{\delta} \frac{(-b_2)(a_1 e^{a_1 b} - a_2 e^{a_2 b})}{(a_1-b_2)e^{a_1 b}+ (b_2-a_2)e^{a_2 b}}  + \lambda \frac{2\frac{\mu}{\sigma^2}(\mu-L_{max})}{-(\mu-L_{max})e^{\frac{2\mu}{\sigma^2}b}-L_{max}}-1, & \text{if } \mu>L_{max}\\
 \frac{L_{max}}{\delta} \frac{(-b_2)(a_1 e^{a_1 b} - a_2 e^{a_2 b})}{(a_1-b_2)e^{a_1 b}+ (b_2-a_2)e^{a_2 b}} -1, & \text{if }\mu \leq L_{max}.
        \end{array}\right.
 \end{align}
Therefore, the threshold level tends to the solution of the equation $G_0(b)=0$, as $\beta$ tends to zero. Moreover, if $\mu\leq L_{max}$, we arrive at the optimal threshold of the classical problem 
\begin{equation}\label{classicalthreshold}
\bar{b}=\frac{1}{a_1-a_2}\ln\left(\frac{a_2(b_2-a_2)}{a_1(b_2-a_1)}\right),
\end{equation}
since ruin is for sure. This corresponds to the same control strategy as if $\lambda$ equals zero. Furthermore, the statement of Lemma \ref{lem1} also holds true for the special case $\beta \to 0$, which just means that $G$ is replaced by $G_0$.
\end{remark}

\begin{lemma}\label{lem1}
If $$(-b_2)\frac{L_{max}}{\delta}+\lambda d_2 -1 > 0$$ then G(b)=0 admits a unique positive solution denoted by $b^*$.
\end{lemma}
\begin{proof}
Analogously as in \cite[p.~6, Lemma 4.1]{Zhao2014}, we know that $G(0)=(-b_2)L_{max}/{\delta}+\lambda d_2 -1>0$ by assumption and $G$ is strictly monotonically decreasing since $G'(b)<0$ for all $b\geq 0$. This can be seen using $\lambda<0$. Finally, we have that $$\lim\limits_{b\to \infty}G(b)=\frac{a_1 b_2}{b_2-a_1} \frac{L_{max}}{\delta}-1<0,$$
since $$\frac{a_1 b_2}{b_2-a_1} \frac{L_{max}}{\delta}-1=\frac{a_1 (-b_2)}{a_1-b_2}\left(\frac{L_{max}}{\delta}+ \frac{1}{b_2}-\frac{1}{a_1}\right)<0.$$ 
The last inequality is true because the first factor is positive and $\left(L_{max}/{\delta}+ 1/{b_2}-1/{a_1}\right)<0$, see also \cite[p.~4, Lemma 2.1]{AsmussenTaksar1997}. Altogether, we obtain the desired result.
\end{proof}
Finally, we get the following candidate solution:
\begin{align}\label{explicitvaluefunction}
\nu(r,t,x)=\left\{\begin{array}{lr}
         e^{-\delta(t-r)} A_1(b^*) (e^{a_1 x} - e^{a_2 x}) \\
         + e^{-\beta(t-r)}(-\lambda + C_1(b^*) e^{c_1 x} + (\lambda-C_1(b^*)) e^{c_2 x})\\
          + \lambda(1-\alpha) , & \text{if } 0\leq x< b^*\\
       e^{-\delta(t-r)} (\frac{L_{max}}{\delta} + B_2(b^*) e^{b_2 x}) \\
       +  e^{-\beta(t-r)}(-\lambda +  D_2(b^*) e^{d_2 x}) + \lambda(1-\alpha), & \text{if } x\geq b^*,
        \end{array}\right.
\end{align}
with $\nu(r,t,x)=\nu(r,t,x;b^*) \in \mathcal{C}^{0,1,1}\left([0,\infty) \times [0,\infty) \times[0,\infty)\right)$ on the whole domain since the coefficients are chosen such that \eqref{smoothfitcond1} and \eqref{smoothfitcond2} are fulfilled.

\begin{theorem}\label{thm:sol} For the time-inconsistent stochastic optimal control problem with objective \eqref{objective} and state process \eqref{stateprocess} the function $\nu$ defined in \eqref{explicitvaluefunction} solves the equilibrium HJB equation in Definition \ref{HJB} in case that 
\begin{equation}\label{parametercase1}
(-b_2)\frac{L_{max}}{\delta}+\lambda d_2 -1 > 0.
\end{equation}
Furthermore, the function $V(t,x)=\nu(t,t,x)$ is an equilibrium value function where $b^*$ is the unique positive solution of the equation $G(b)=0$ and 
\begin{align}\label{eqbarrier}
\hat{l}(t,x)=\left\{\begin{array}{lr}
        0, & \text{if } 0\leq x<b^*,\\
        L_{max}, & \text{if } x\geq b^*,
        \end{array}\right.
\end{align}
is the equilibrium control according to Definition \ref{eqlaw}.
\end{theorem}

\begin{proof}
We have to make sure that $\nu$ in \eqref{explicitvaluefunction} and $\hat{l}$ in \eqref{eqbarrier} satisfy the assumptions for the usage of Theorem \ref{verification}. Using the explicit representation \eqref{eqbarrier} we immediately observe that $\hat{l} \in  \Theta$. Hence, we start with showing that $\nu$ in \eqref{explicitvaluefunction} together with \eqref{smoothfitcond1}, \eqref{smoothfitcond2} and \eqref{smoothfitcond3} solves the extended HJB equation. For that reason it is straightforward to insert $\nu$ into the equation \eqref{HJBbarrier} and observe that it is a solution (using \eqref{smoothfitcond1} and  \eqref{smoothfitcond2}).\\
Further, using the assumption \eqref{parametercase1}, Lemma \ref{lem1} yields that there exists a unique positve $b^*$ such that \eqref{smoothfitcond3} is fulfilled. Moreover, this gives that $\frac{\partial^2}{\partial x^2} \nu (t,t,b^*-)=\frac{\partial^2}{\partial x^2} \nu (t,t,b^*+)$. 
Using \eqref{HJBbarrier} at $(t,t,b^*-)$ and $(t,t,b^*+)$ yields that $-\frac{\sigma^2}{2}\frac{\partial^2}{\partial x^2} \nu (t,t,b^*-)=\mu \frac{\partial}{\partial x} \nu (t,t,b^*-) - \lambda \beta$ and $-\frac{\sigma^2}{2}\frac{\partial^2}{\partial x^2} \nu (t,t,b^*+)=(\mu-L_{max}) \frac{\partial}{\partial x} \nu (t,t,b^*-)+L_{max} - \lambda \beta$. Together with \eqref{smoothfitcond3}, we obtain that $\nu(t,t,x)$ is twice continuously differentiable. Summing up, we have that $\nu(r,t,x) \in \mathcal{C}^{0,1,1}\left([0,\infty)^3 \right) \ \cap \ \mathcal{C}^{0,1,2}\left([0,\infty) \times [0,\infty) \times ([0,\infty)\backslash \{b^*\}) \right)$ in addition to the fact that $\nu(t,t,x)$ is twice continuously differentiable in $x$.\\
\\
Furthermore, we show that $\nu(t,t,x)$ is concave for $x\in [0,\infty)$, \  i.e. $\frac{\partial^2}{\partial x^2} \nu (t,t,x)<0$.
First of all, we immediately observe that $\frac{\partial^2}{\partial x^2} \nu (t,t,x)<0$ for $x\geq b^*$, since $B_2(b^*)<0$ and $D_2(b^*)<0$.
On the other hand, we have that $$\frac{\partial^3}{\partial x^3}V_2(b^*;b^*)= (-\lambda)e^{(c_1+c_2)b^*}(c_1-c_2) \frac{c_1c_2(c_1+c_2)-(c_1^2+c_1c_2+c_2^2)d_2}{^(c_1-d_2)e^{c_1b^*}+(d_2-c_2)e^{c_2b^*}}>0,$$ and since $\frac{\partial^4}{\partial x^4}V_2(x;b^*)<0$ for $0\leq x< b^*$, we obtain that $\frac{\partial^3}{\partial x^3}V_2(x;b^*)>0$ for $0\leq x< b^*$. This, together with $\frac{\partial^3}{\partial x^3}V_1(x;b^*)>0$, yields that $\frac{\partial^3}{\partial x^3}\nu(t,t,x)>0$ for $0\leq x< b^*$. This in turn implies for $0\leq x< b^*$ that  $\frac{\partial^2}{\partial x^2}\nu(t,t,x)< \frac{\partial^2}{\partial x^2}\nu(t,t,b^*)$. Since $\nu(t,t,x)$ is twice continuously differentiable in $x$ we obtain $\frac{\partial^2}{\partial x^2}\nu(t,t,b^*-)=\frac{\partial^2}{\partial x^2}\nu(t,t,b^*+)<0$. Finally, this yields that $\nu(t,t,x)$ is concave.\\

We already know that $\frac{\partial}{\partial x}\nu(t,t,b^*)=1$ for $b^*>0$. Combining this with concavity, we obtain that $\frac{\partial}{\partial x}\nu(t,t,b^*)>1$ for $0\leq x< b^*$ and $\frac{\partial}{\partial x}\nu(t,t,b^*)\leq 1$ for $x\geq b^*$. On top of that, the function $\nu(t,t,x)$ is strictly monotonically increasing since $\frac{\partial}{\partial x}\nu(t,t,x)>0$. For the case $0\leq x< b^*$ this holds because the first derivative is even larger than one. For $x\geq b^*$ we have the desired monotonicity from $\frac{\partial}{\partial x}\nu(t,t,x)= B_2(b^*) b_2 e^{b_2x}+ D_2(b^*) d_2 e^{d_2x}>0$.\\
\\
Altogether, the behaviour of the first and second derivatives of $\nu(t,t,x)$  yields  that, if this function fulfils the equation \eqref{HJBbarrier}, it solves equations \eqref{HJB1} and \eqref{HJB2}. On top of that $\nu(r,t,x)$ is bounded in $t$ and $x$, hence we have that \eqref{HJB4} implies  \eqref{HJB3} as mentioned in Remark \ref{growthcond}. Since $$\lim\limits_{\substack{x\to\infty \\ t\to \infty}} \nu(r,t,x)= \lambda(1-\alpha)$$ for every fixed $r \in [0,\infty)$ we have:
\begin{align*}
&\lim\limits_{w\to \infty} \E_{t,x}\left[ \nu(r,w\wedge \tau^{\hat{l}}, X^{\hat{l}}_{w\wedge \tau^{\hat{l}}})\right]\\
&=\lim\limits_{w\to \infty} \E_{t,x}\left[ \nu(r,w\wedge \tau^{\hat{l}}, X^{\hat{l}}_{w\wedge \tau^{\hat{l}}})  1_{\{\tau^{\hat{l}}< \infty\}} +\nu(r,w\wedge \tau^{\hat{l}}, X^{\hat{l}}_{w\wedge \tau^{\hat{l}}})  1_{\{\tau^{\hat{l}}= \infty\}} \right]\\ 
&=\E_{t,x}\left[\lambda \left(1- \alpha\right)   1_{\{\tau^{\hat{l}}< \infty\}} +\lambda \left(1- \alpha\right) 1_{\{\tau^{\hat{l}}= \infty\}} \right]\\ 
&=\lambda \left(1- \alpha\right).
\end{align*}
The last step, for verifying that $\nu$ and $\hat{l}$ fulfil the requirements of the verification theorem, is to show that $\nu$ and its relevant derivatives are bounded. Since,
\begin{equation*}\frac{\partial}{\partial x}\nu(r,t,x) \leq e^{(\delta \vee \beta) r} e^{-(\delta \wedge \beta) t} \frac{\partial}{\partial x}\nu(t,t,x) \leq e^{(\delta \vee \beta) r} \frac{\partial}{\partial x}\nu(t,t,0),
\end{equation*}
and $0<\nu(t,t,0)<\infty $ is a constant, we obtain that the first derivative in $x$ is bounded. Furthermore, we observe, using $b_2<0$ and $d_2<0$, that  $\nu$ is bounded, and by building the respective derivatives that also $\frac{\partial^2}{\partial x^2}\nu(r,t,x)$ and $\frac{\partial}{\partial t}\nu(r,t,x)$ are bounded.\\
Altogether, we can apply Theorem \ref{verification} and obtain that $\nu$ is the equilibrium value function and $\hat{l}$ is the associated equilibrium control.
\end{proof}

Analogous to Theorem \ref{thm:sol} we can deal with the remaining parameter constellation.
\begin{proposition}
If $(-b_2)L_{max}/{\delta}+\lambda d_2 -1 \leq 0,$ then $b^*=0$ and the equilibrium control becomes $\hat{l}(t,x)=L_{max}$.
The function $\nu$ from \eqref{explicitvaluefunction} now reads as 
\begin{equation*}
\nu(r,t,x)=e^{-\delta(t-r)}\frac{L_{max}}{\delta}(1- e^{b_2 x})+e^{-\beta(t-r)}\lambda(e^{d_2 x}-1)+\lambda(1-\alpha).
\end{equation*}
This function solves the extended HJB equation in Definition \ref{HJB} and $V(t,x)=\nu(t,t,x)$ is an equilibrium value function associated to $\hat{l}(t,x)=L_{max}$.
\end{proposition}

\begin{proof}
We proceed exactly in the same way as in the proof for Theorem \ref{thm:sol}. It can be verified that $\nu$ solves the equation \eqref{HJBbarrier} with $b=b^*=0$. Furthermore, since $\frac{\partial}{\partial x} \nu(r,t,x)>0$ and $\frac{\partial^2}{\partial x^2} \nu(r,t,x)<0$ for all $(r,t,x) \in [0,\infty) \times[0,\infty) \times[0,\infty)$ in addition to $\frac{\partial}{\partial x} \nu(r,t,0)\leq 1$ for all $(r,t) \in \{(q,s)\in \R^2|0\leq q\leq s < \infty\}$ we obtain that  $$\frac{\partial}{\partial x} \nu(r,t,x)\leq 1$$ for all $(r,t) \in \{(q,s)\in \R^2|0\leq q\leq s < \infty\}$ and $ x \in [0,\infty)$.
Hence, $\nu$ and $\hat{l}(t,x)$ fulfil \eqref{HJB1} and \eqref{HJB2}. Moreover, $\nu$ is bounded and satisfies the correct boundary condition, which yields that it fulfils \eqref{HJB3}. Finally, $\nu \in \mathcal{C}^{0,1,2}\left( [0,\infty)^3\right)$ and it satisfies the extended HJB equation. Since $\nu$, $\frac{\partial}{\partial t} \nu$, $\frac{\partial}{\partial x} \nu$ and $\frac{\partial^2}{\partial x^2} \nu$ are all bounded in $t$ and $x$, the requirements of Theorem \ref{verification} are fulfilled. Consequently, $\nu$ is the equilibrium value function and $\hat{l}(t,x)$ the associated equilibrium control.
\end{proof}

\subsection{Parameter areas}
The arising inequality in Theorem \ref{thm:sol} splits the parameter space into different areas, such that this classification gives rise to an interpretation. First of all, we know that the classical dividend threshold $\bar{b}$ is positive if $(-b_2)L_{max}/{\delta} -1>0$ which is equivalent to 
\begin{equation}\label{inequality_interpretation}
\frac{\frac{\sigma^2}{2}}{\mu}< \frac{L_{max}}{\delta},
\end{equation}
see Schmidli (2008). Otherwise, if \eqref{inequality_interpretation} is not fulfilled, we have that $\bar{b}=0$. Furthermore, note that by the properties of $G(b)$ in \eqref{eq:G} we obtain that the equilibrium dividend threshold $b^*$ fulfills $b^*\geq \bar{b}$.
As a matter of fact, we obtain the following cases, which are also illustrated in Table \ref{tab:inequality}.
\begin{enumerate}
\item If \eqref{inequality_interpretation} is fulfilled and  $\bar{b}>0$, then we obtain also that $b^*>0$. The condition \eqref{inequality_interpretation} can be viewed as an advantageous parameter configuration which motivates the controller to take less risk and to payout dividends only at a surplus level larger than zero in order to prevent ruin for a longer time. The prospect is advantageous since the maximal reachable value of dividends (the right-hand side of \eqref{inequality_interpretation})  is relatively large compared to the variance of the uncontrolled surplus divided by its mean (the left-hand side of \eqref{inequality_interpretation}). Note that $\E(X_t)=\mu t$ and $Var(X_t)=\sigma^2 t$. Moreover, we observe that if $\lambda$ decreases, then the threshold level increases since one wants to prevent ruin and the associated larger penalty even more.

\item Whereas if \eqref{inequality_interpretation} is not fulfilled and hence $\bar{b}=0$, the future prospect is such that the controller in the classical problem ($\lambda=0$) would pay out dividends at any surplus level. Especially if $\mu-L_{max}$ is negative, this strategy can be interpreted as the ruin strategy since ruin is certain in this case. But given that a penalty exists ($\lambda<0$), in case of ruin, we can  distinguish the following two cases:
\begin{itemize}
\item[i)] If $\lambda <\Lambda$ for 
$$\Lambda\coloneqq\left(1+ b_2 \frac{L_{max}}{\delta}\right)\frac{1}{d_2},$$ 
the impact of the penalty is too severe and we obtain that $b^*>0$. This yields that dividends are paid out only for positive surplus levels as in the first case. We observe that the penalty in case of ruin has a stronger deterrent effect and therefore one wants to prevent ruin.

\item[ii)] On the other hand, if $\lambda \geq \Lambda$  meaning that the weight on the penalty is moderate, we obtain that also $b^*=0$. In this case the penalty is not severe enough and the equilibrium strategy suggests to pay out dividends at any surplus level even if it is more likely that ruin occurs compared to the use of a positive threshold level.
\end{itemize}
\end{enumerate}

\begin{table}[ht]
\begin{center}
\begin{tabular}{ccr}
\toprule
$\displaystyle (-b_2)\frac{L_{max}}{\delta} -1> 0 \quad \Rightarrow \quad \bar{b}>0$ & & $\Rightarrow \quad b^*>0$ \\
\midrule
\multirow[c]{2}{*}[-2pt]{$\displaystyle (-b_2)\frac{L_{max}}{\delta} -1\leq 0 \quad \Rightarrow \quad \bar{b}=0$} & & if $\lambda<\Lambda \quad \Rightarrow \quad b^*>0$ \\
\cmidrule{3-3}
 & & if $\lambda \geq \Lambda \quad \Rightarrow \quad b^*=0$ \\
\bottomrule
\end{tabular}
\caption{The nature of the strategy depending on the parameter configuration.}\label{tab:inequality}
\end{center}
\end{table}

\section{Meeting the constraint}\label{sec:constraint}

In the following we discuss the motivating problem, i.e. dividend maximization including the constraint \eqref{constraint}. We try to find a solution which fulfils the constraint  for fixed $(t,x)$, this means in a \emph{precommitment sense}. Up to now we have found for every fixed $\lambda \leq 0$ an equilibrium threshold strategy $b^*_\lambda$ and an associated value function. If
\begin{equation}\label{Lambda}
\lambda< \left(1+ b_2 \frac{L_{max}}{\delta}\right)\frac{1}{d_2}=\Lambda
\end{equation}
there exists a unique equilibrium threshold level $b^*_\lambda >0$ for the associated time-inconsistent problem. If $\lambda\geq  \Lambda$ the equilibrium threshold is $b^*_\lambda=0$. Concerning \eqref{Lambda} note that $d_2<0$.\\
In the case $\beta>0$, we observe from equation \eqref{eq:G}, $G(b)=0$, that if $\lambda \to -\infty$, $b$ tends to $+\infty$.
The factor next to $\lambda$ is negative and consequently monotonically increasing in $b$, so this term vanishes for  $b \to \infty$.
The same holds true if $\beta =0$ and $\mu > L_{max}$.\\
On the other hand, if $\beta=0$ and $\mu \leq L_{max}$, equation \eqref{eq:G0}, now $G_0(b)=0$, yields that $b^*=\bar{b}$, where $\bar{b}$ is the classical threshold level.
Hence, we consider in the following $\beta>0$. The case $\mu > L_{max}$ and $\beta \to 0$ can be analogously treated.\\
But, for $\mu \leq L_{max}$ and $\beta \to 0$, we have for all $\lambda \leq 0$ that $b^*=\bar{b}$. Therefore, the ruin probability is one in this case.
This yields that the constraint cannot be fulfilled and the equilibrium threshold level coincides with the classical threshold level.\\
Focusing on the constraint \eqref{constraint}, we can rewrite $V_2(x;b^*_\lambda)$, which is constructed to solve \eqref{PDE2},
\begin{equation}\label{V2}
\tilde{V}_2(x;b^*_\lambda):=\frac{1}{\lambda}\left(V_2(x;b^*_\lambda)+\lambda \right)=\E_{t,x}\left[e^{-\beta(\tau^{\hat{l}^\lambda}-t)} \right].
\end{equation}
For an arbitrary threshold level $b\geq0$ we define:
\begin{equation}\label{Laplacetransform}
w(x,b):=\E_{t,x}\left[e^{-\beta(\tau^{b}-t)} \right]\left(=\E_{0,x}\left[e^{-\beta \tau^{b}} \right]\right).
\end{equation}
Here $\tau^b=\inf\{s\geq t| X^b_s \leq 0\,, X^b_t=x\}$ and $X^b$ is the process \eqref{stateprocess} with $dL_s=L_{max}1_{\{X^b_s \geq b  \}}ds$.
A solution to equation \eqref{PDE2}, incorporating \eqref{smoothfitcond1} and \eqref{smoothfitcond2}, for $b\geq 0$ has the form $(\lambda w(x,b)-\lambda)$. Hence, we have,
\begin{align*}
w(x,b)=\left\{\begin{array}{lr}
\frac{(d_2-c_2) e^{-c_1 (b-x)}+(c_1-d_2)e^{-c_2 (b-x)}}{(d_2-c_2)e^{-c_1 b}+(c_1-d_2) e^{-c_2 b}}
     , & \text{if } 0\leq x< b,\\
      \frac{(c_1-c_2) e^{d_2 (x-b)}}{(d_2-c_2)e^{-c_1 b}+(c_1-d_2) e^{-c_2 b}}, & \text{if } x\geq b.
        \end{array}\right.
\end{align*}
This can be seen using a Feynman-Kac type argument. According to \eqref{smoothfitcond1} and \eqref{smoothfitcond2}, $w(x,b)$ is continuously differentiable in $x$. For $b_\lambda^*$ the equilibrium threshold with $G(b_\lambda^*)=0$, we have $\tilde{V}_2(x;b^*_\lambda) = w(x,b_\lambda^*)$.
On top of this, if $\mu> L_{max}$, we obtain $$\lim\limits_{\beta\to 0}w(x,0)=\psi(x,L_{max}),$$ from \eqref{ruinprob}.

\subsection{Determining the matching \texorpdfstring{$\lambda$}{TEXT} for \texorpdfstring{$\beta>0$}{TEXT}}
In the following, we can present two propositions which are based on Proposition 3.2 and Proposition 4.1 from \cite{Hernandez2015}. Certainly, our approach also covers the particular situation of $\beta =\delta$. Therefore, we obtain similar results as in \cite{Hernandez2015}, but for the diffusion model with bounded dividend rate.

\begin{proposition}\label{prop:duality1}
If $\Lambda\leq 0$, we have that for all $b>0$ there exists a unique $\lambda < \Lambda$ such that $b$ is the equilibrium threshold level for \eqref{eqbarrier} with parameter $\lambda$.\\
On the other hand, if $\Lambda> 0$, we have that for all $b\geq \bar{b}$, (where $\bar{b}$ is specified in \eqref{classicalthreshold}), there exists a unique $\lambda \leq 0$ such that $b$ is the equilibrium threshold level for \eqref{eqbarrier} with parameter $\lambda$.
\end{proposition}

\begin{proof}
Since the relation between $\lambda$ and the equilibrium threshold results from equation \eqref{eq:G}, we define for a given $b\geq0$ the corresponding value of $\lambda$ to be
\begin{equation*}
\lambda(b):=\frac{1-N_1(b)}{N_2(b)}.
\end{equation*}
Where we have that $$N_1(b)=\frac{L_{max}}{\delta} \frac{(-b_2)(a_1 e^{a_1 b} - a_2 e^{a_2 b})}{(a_1-b_2)e^{a_1 b}+ (b_2-a_2)e^{a_2 b}}$$ and $$N_2(b)=\frac{(c_2-c_1) d_2 e^{(c_1+c_2)b}}{(d_2-c_1)e^{c_1b}+(c_2-d_2)e^{c_2 b}}.$$ For these two functions we have that $N_1(b)>0$, $N'_1(b)<0$, $N_2(b)<0$ and  $N'_2(b)>0$.\\
Moreover, if $\Lambda\leq 0$, then $1+b_2 L_{max}/{\delta}=1-N_1(0)\geq0$, which in turn gives that $1-N_1(b)\geq 0$.
Finally, this yields that $\lambda'(b)<0$ and we obtain the first statement, since $\lambda(0)=\Lambda\leq 0$ and $\lambda(b) \to -\infty$ for $b\to \infty$.\\
For the second statement let $\Lambda> 0$, we obtain that $1-N_1\left(\bar{b}\right)=0$ and since $N_1(b)$ is monotonically decreasing, we get that $1-N_1(b)\geq 0$ for all $b\geq \bar{b}$. This along the lines of the classical dividend problem.
As before,  we see that $\lambda'(b)<0$ for all $b\geq \bar{b}$. This, in combination with $\lambda\left(\bar{b}\right)=0$ and $\lambda(b) \to -\infty$ for $b\to \infty$, proves the second assertion.
\end{proof}

\begin{remark}
For fulfilling the required constraint, we have to choose the pair $(b^*,\lambda^*)$ dependent on the current state $x$ (setting $t=0$).
This means that we have a solution of the constrained problem in a precommitment sense. This notion is well explained in \cite{Delong}.
Since we cannot fulfil the constraint $w(x,b)\leq \alpha$ for a finite $b>0$ if $x \leq \bar{x}= \ln(\alpha)/{c_2}$, we only consider the case $x>\bar{x}$. In the special case $x = \ln(\alpha)/{c_2}$ we are only able to meet the constraint if we never pay out dividends, which corresponds to a threshold at infinity.
\end{remark}

\begin{proposition}\label{prop3}
For every given $0<\alpha\leq1$ and every $x>\bar{x}$ there exists a pair $(b^*,\lambda^*)$ with $b^*\geq 0$ and $\lambda^*\leq 0$, such that $b^*$ is the equilibrium threshold level for \eqref{eqbarrier} with parameter $\lambda^*$ and further for this triple $(x,b^*,\lambda^*)$ it holds that 
\begin{align}\label{dual1}
w(x,b^*)\leq \alpha,
\end{align}
and
\begin{align}\label{dual2}
\lambda^*(w(x,b^*)-\alpha)=0.
\end{align}
\end{proposition}

\begin{proof}
In the following we consider the different cases:\\
 1.) $ \Lambda \leq 0$ with 1a) $ w(x,0)\leq \alpha$ and 1b) $ w(x,0)>\alpha$; on the other hand 2.) $ \Lambda > 0$ with 2a) $ w(x,\bar{b})\leq \alpha$ and 2b) $ w(x,\bar{b})>\alpha$. Moreover, note that  since $x>\bar{x}$ we always have that $\alpha > \lim\limits_{b\to \infty} w(x,b)=e^{c_2 x}$.
Firstly, let $\Lambda\leq 0$ and $w(x,0)\leq \alpha$, then we can take the equilibrium threshold to be $b^*=0$ since the constraint is already fulfilled. Moreover, $\lambda^*$ has to be zero, such that the second part of the objective function vanishes.
Since $b^*=0$, all choices of $\lambda \in [\Lambda,0]$ yield an equilibrium threshold at level zero. Hence, we can take $\lambda^*=0$.\\
If $w(x,0)> \alpha$, and consequently $$w(x,0)> \alpha> \lim\limits_{b\to \infty} w(x,b)=e^{c_2 x},$$ we have  that there exists a unique $b^*>0$ such that $w(x,b^*)=\alpha$ (remember $\frac{\partial}{\partial b} w(x,b)<0$). Using Proposition \ref{prop:duality1}, there exists a unique $\lambda^*<\Lambda$ such that $b^*$ is an equilibrium threshold.\\
Secondly, if $\Lambda>0$ and $w(x,\bar{b})\leq \alpha$, we have to take $b^*=\bar{b}$ such that we can meet the constraint.
As before we can set $\lambda^*=0$ since $\bar{b}$ is the equilibrium threshold level for $\lambda^*=0$. On the other hand, if $w(x,\bar{b})> \alpha$ and hence $$w(x,\bar{b})> \alpha> \lim\limits_{b\to \infty} w(x,b)=e^{c_2x},$$ then there exists a unique $b^*>\bar{b}$ such that $w(x,b^*)=\alpha$.  Due to  Proposition \ref{prop:duality1}, there exists a unique  $\lambda^* <0$, such that $b^*$ is an associated equilibrium threshold.
\end{proof}
Since we have revealed an equilibrium control strategy and the associated equilibrium value function for the objective \eqref{objective}, we cannot use the properties of a maximal reward function as in \eqref{valuefunction}. Hence, standard duality statements do not apply.\\
We summarize our achievements as follows. For a given ruin constraint level $0<\alpha\leq 1$ we can identify the correct threshold level $b^*=b^*(\alpha)$, such that the constraint \eqref{constraint} is fulfilled for a given initial state $x_0>\bar{x}$. Subsequently, we can specify a value $\lambda^*$, such that $b^*$ is the threshold level of an equilibrium control strategy \eqref{eqbarrier} for the equilibrium value function $V(t,x_0)=\nu(t,t,x_0)$ as in \eqref{explicitvaluefunction}.\\
Notice that the equilibrium threshold level and the associated $\lambda^*$ is depending on the current state $x_0$, in such a way that \eqref{dual1} and \eqref{dual2} for this $\alpha$ is only fulfilled at $x_0$.\\
In contrast to \cite{Hernandez2015}, we only obtain one part of the analogous duality statements. Let $\hat{l}^\lambda$ be the equilibrium control for a problem with general parameter $\lambda$, and $\hat{l}^{\lambda^*}$ the corresponding equilibrium control such that the constraint for $\alpha$ is fulfilled in $x_0>\bar{x}:$
\begin{multline*}
\inf_{\lambda\leq 0} J(t,x_0,\hat{l}^\lambda) \leq J(t,x_0,\hat{l}^{\lambda^*})= 
 \E_{t,x_0}\left[\int\limits_t^{\tau^{\hat{l}^{\lambda^*}}} e^{-\delta(s-t)}\hat{l}^{\lambda^*}(s,X_s^{\hat{l}^{\lambda^*}})ds+\lambda^* \left(e^{-\beta(\tau^{\hat{l}^{\lambda^*}}-t)} - \alpha\right) \right]\\
=\E_{t,x_0}\left[\int\limits_t^{\tau^{\hat{l}^{\lambda^*}}} e^{-\delta(s-t)}\hat{l}^{\lambda^*}(s,X_s^{\hat{l}^{\lambda^*}})ds\right]\leq \sup_{L } \E_{t,x_0}\left[\int\limits_t^{\tau^L} e^{-\delta(s-t)} dL_s \right].
\end{multline*}
The last equality holds true, because the penalty term vanishes in $x_0$ for $\hat{l}^{\lambda^*}$ as in \eqref{dual2}.\\
\\
\begin{remark}
The functions $\tilde{V}_2$ and $w$ defined in \eqref{V2} and \eqref{Laplacetransform} characterize the evolution of the constraint, or more precisely, of the level $\alpha$. Suppose that the level $\alpha$, the initial surplus $x_0>\bar{x}$ and the corresponding equilibrium strategy $\hat{l}^{\lambda^*}$ with threshold height $b^*_{\lambda^*}(=b^*(x_0))$ are initially fixed. Then, at some point in time $t>0$ with $\tau^{\hat{l}^{\lambda^*}}(0,x_0)>t$,
the constraint level
\begin{equation*}
\alpha_t:=e^{-\beta t}w\left(X_t^{\hat{l}^{\lambda^*}},b^*(x_0)\right),
\end{equation*}
is reached for $\left(t,X_t^{\hat{l}^{\lambda^*}}\right)$ and original $(x_0,\,\alpha)-$admissible strategy $\hat{l}^{\lambda^*}$. From the Markov property we get that
$\left(\alpha_{t\wedge\tau^{\hat{l}^{\lambda^*}}(0,x_0)}\right)_{t\geq 0}$ is an $\{\mathcal{F}_t\}_{t\geq 0}$-martingale. Consequently, the strategy $\hat{l}^{\lambda^*}$ is the particular equilibrium control for which the constraint level $\alpha_t$ is sharp if the process is restarted at $X_t^{\hat{l}^{\lambda^*}}$. This means that the quality of the strategy is preserved if the constraint is adapted. The martingale property of $\left(\alpha_{t\wedge\tau^{\hat{l}^{\lambda^*}}(0,x_0)}\right)_{t\geq 0}$ shows that these adaptions are \textit{consistent over time}.
\end{remark}

\subsection{Quality of the solution}
As mentioned above we do not obtain classical duality results, nevertheless using our explicit  form we are able to optimize within the class of threshold strategies. Namely, we can rewrite the constrained optimization problem
\begin{equation}\label{constrainedoptimizationproblem}
\sup_{b\geq 0} V_1(x;b), \quad \text{s.t. } w(x,b)\leq \alpha \quad \text{ for fixed } x \in (\bar{x},\infty).
\end{equation} 
Using that $\bar{b}$ is the maximizing argument of the classical problem, together with $\frac{\partial}{\partial b} w(x,b)<0$ for all $b\geq 0$, we obtain for an arbitrary but fixed $x>\bar{x}$ that \eqref{constrainedoptimizationproblem} simplifies to
\begin{equation*}
\sup_{b\geq \tilde{b}} V_1(x;b).
\end{equation*} 
Whereas $\tilde{b}$ is defined as follows
\begin{equation}\label{btilde}
\tilde{b} = \begin{cases}
\bar{b}, \text{ if } w(x,\bar{b})\leq \alpha,\\
b^\prime, \text{ where } w(x,b^\prime)= \alpha.
\end{cases}
\end{equation}
Finally, the optimal threshold of the constrained problem will be this $\tilde{b}$. Since $\tilde{b}\geq \bar{b}$ and $\frac{\partial}{\partial b} V_1(x;b)<0$ for all $b > \bar{b}$. Note, that the optimal threshold $\tilde{b}$ depends on the initial value $x$. If $x\leq \bar{x}$ the constraint cannot be fulfilled.\\
Recalling the proof of Proposition \ref{prop3}, we observe that our chosen equilibrium threshold of the penalized problem coincides with the optimal threshold $\tilde{b}$. This just means that if $w(x,\bar{b}) \leq \alpha$, then we can take $b^*=\bar{b}(=\tilde{b})$ which is the optimal threshold of the constrained problem and $\lambda^*=0$. This yields $\nu(t,t,x;b^*)=V_1(x;b^*)$.\\
If $w(x,\bar{b}) > \alpha$, then we take $b^*=\tilde{b}(>\bar{b})$, which is again the optimal threshold in the constrained problem. In this case, we have to take $\lambda^*<0$ for the penalized reward function such that $b^*$ corresponds to an equilibrium control, but for the initially fixed $x>\bar{x}$. Again we obtain $\nu(t,t,x;b^*)=V_1(x;b^*)$ by Proposition \ref{prop3}.\\
Note, that the constraint is only fulfilled for this initial $x$ and the penalized function $\nu$ will change for a different $x$ since $\lambda^*$ is not zero and depends on $x$. In the end, $\tilde{b}$ from \eqref{constrainedoptimizationproblem} is optimal in the family of threshold strategies, but our analysis shows that for a certain $\lambda^*$ it is also an equilibrium control from the general class of strategies $\Theta$ for the time-inconsistent penalized problem.\\
On top of this, we have computed the optimal threshold for the penalized problem with $\lambda\leq 0$ fixed
\begin{equation}\label{optimumprecommitment}
\sup_{b\geq 0} \nu(t,t,x;b).
\end{equation} 
The optimal threshold will depend on the initial $x$ and the solution is therefore in a  precommitment sense.\\
 Altogether, we illustrate these findings in Figure \ref{fig:areas}, where the dashed red line corresponds to the $x$-dependent optimal threshold of \eqref{optimumprecommitment} for $\lambda = -10$. Furthermore, in Figure \ref{fig:areas} the green area is limited by the green line $\bar{b}$ and by the black line which corresponds to those $b^\prime$ such that $w(x,b^\prime)=\alpha$. Consequently, this boundary of the green area exactly equals our optimal threshold level $\tilde{b}$ as in \eqref{btilde} for different $x$. We can observe that if $x$ is close to $\bar{x}$ we cannot use $\bar{b}$ anymore. Hence, we have to take a threshold which corresponds to the equilibrium threshold for a value $\lambda^*<0$. This is for example illustrated by the blue line for the value $\lambda^*=-10$.
\begin{figure}[htb]\centering
\includegraphics[width=0.9\linewidth]{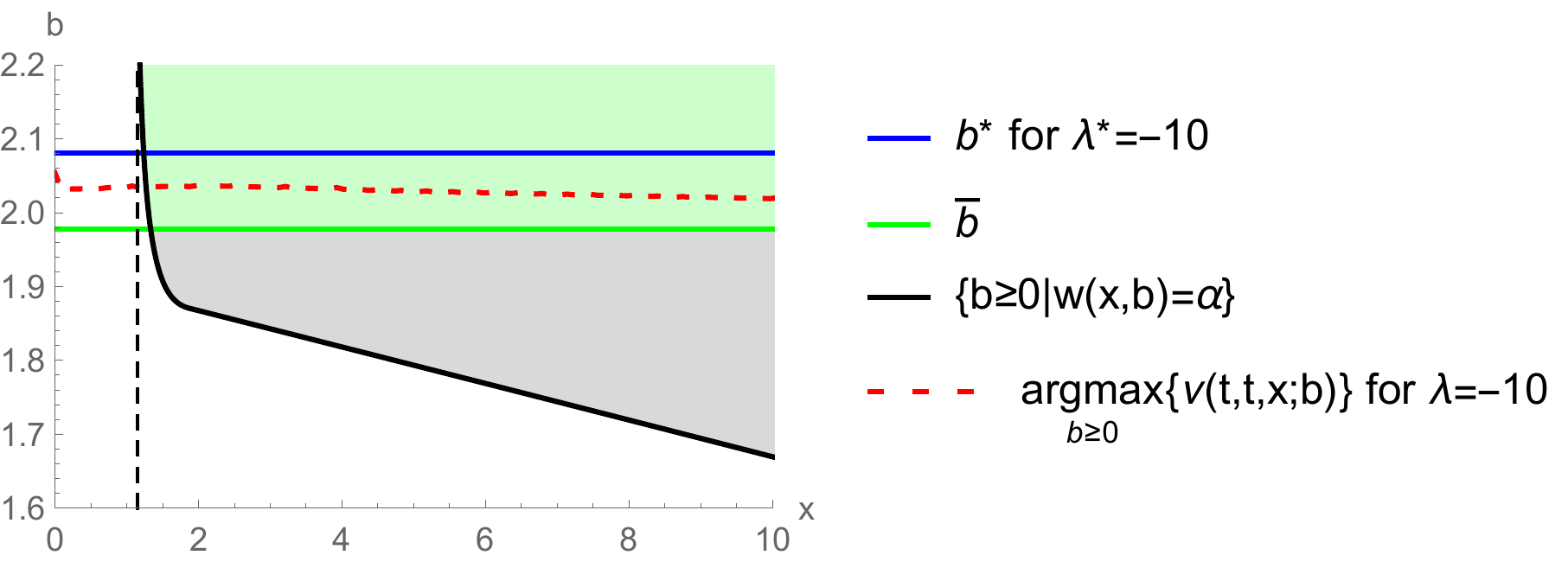}
\caption{The different threshold levels for different initial values. Parameters: $\mu=2$, $\sigma=1$, $\delta=0.1$, $\beta=0.2$, $L_{max}=4$, $\alpha=0.01$.}\label{fig:areas}
\end{figure}

\subsection{Concerning the ruin probability}

As mentioned in Remark \ref{zeroequation}, if we send $\beta$ to zero, there are some differences depending on the relation between $\mu$ and $L_{max}$.
The above procedure can be repeated, if we look at the limit $\beta$ to zero for $\mu>L_{max}$.\\
But, if $\mu\leq L_{max}$, the equation $G_0(b)=0$, see \eqref{eq:G0}, does not depend on $\lambda$ anymore. The second term in $G_0(b)$ vanishes if $\beta$ tends to zero. This yields that for any $\lambda\leq 0$, we obtain the same equilibrium threshold level, namely, the classical threshold $b=\bar{b}$. This is because for any finite threshold strategy, the ruin probability is one. Therefore, $V_2(x)=0$ and $\tilde{V}_2(x)=1$.
Moreover, the classical threshold is the only reasonable solution, since the reward function differs only by a constant from the classical reward function in the time-consistent situation.\\
We are aware of the fact that the reward function could be increased by using a precommitment solution which takes account of the potentially high penalty and has therefore a ruin probability smaller than one. The drawback of such solutions is that one has to restrict (or even specify) the strategies in order to verify the desired optimality in that class of controls.

\section{Graphical illustrations}\label{sec:graph}
In this short section we present a numerical example, which displays the dependence of the equilibrium threshold on the differing discount rate $\beta$ and the parameter $\lambda$. In the corresponding plots one nicely observes that the proposed \emph{type of duality} considerations are highly legitimate. 

\subsection{The case \texorpdfstring{$\mu>L_{max}$}{TEXT}}
In the following, we consider the exemplary parameters displayed in Table \ref{tab:parameters}.
\begin{table}[ht]
\begin{center}
\begin{tabular}{l|l|l|l|l}
$\mu$ & $\sigma$ & $\delta$ & $L_{max}$ & $\alpha$\\
\hline
2 & 1 & 0.1 & 1.9 & 0.5
\end{tabular}
\caption{Parameter set}\label{tab:parameters}
\end{center}
\end{table}
If we do not vary $\beta$ and $\lambda$, we set $\beta=0.2$ and $\lambda=-50$.
The change of the threshold, if we  alter the underlying discount rate $\beta$, is illustrated in Figure \ref{fig:barrier_in_beta_1}. We can observe that if we send $\beta$ to infinity the equilibrium threshold tends to the classical optimal threshold level $\bar{b}$.
 Furthermore, we are able to reproduce the results from \cite{ThonAlb2007}, namely, in our setting we have to take $\beta= \delta$ besides $ \lambda(-\delta)=\tilde{\Lambda}$, where $\tilde{\Lambda}$ denotes the constant in the penalty term of \cite{ThonAlb2007}. The corresponding optimal control threshold is denoted by $b(\delta)$.\\
On the other hand, the dependence on $\lambda$ for a fixed  value of $\beta$ is depicted in Figure \ref{fig:barrier_in_lambda_1}.
The behaviour of the equilibrium threshold for different values of $\lambda$ is similar for positive $\beta$ and for the limit $\beta$ to zero. 
Clearly, this is due to the fact that the equation $G_0(b)=0$ still depends on $\lambda$ if $\mu>L_{max}$, as already mentioned in Remark \ref{zeroequation}.
 
\begin{figure}[htb]\centering
\includegraphics[width=0.75\linewidth]{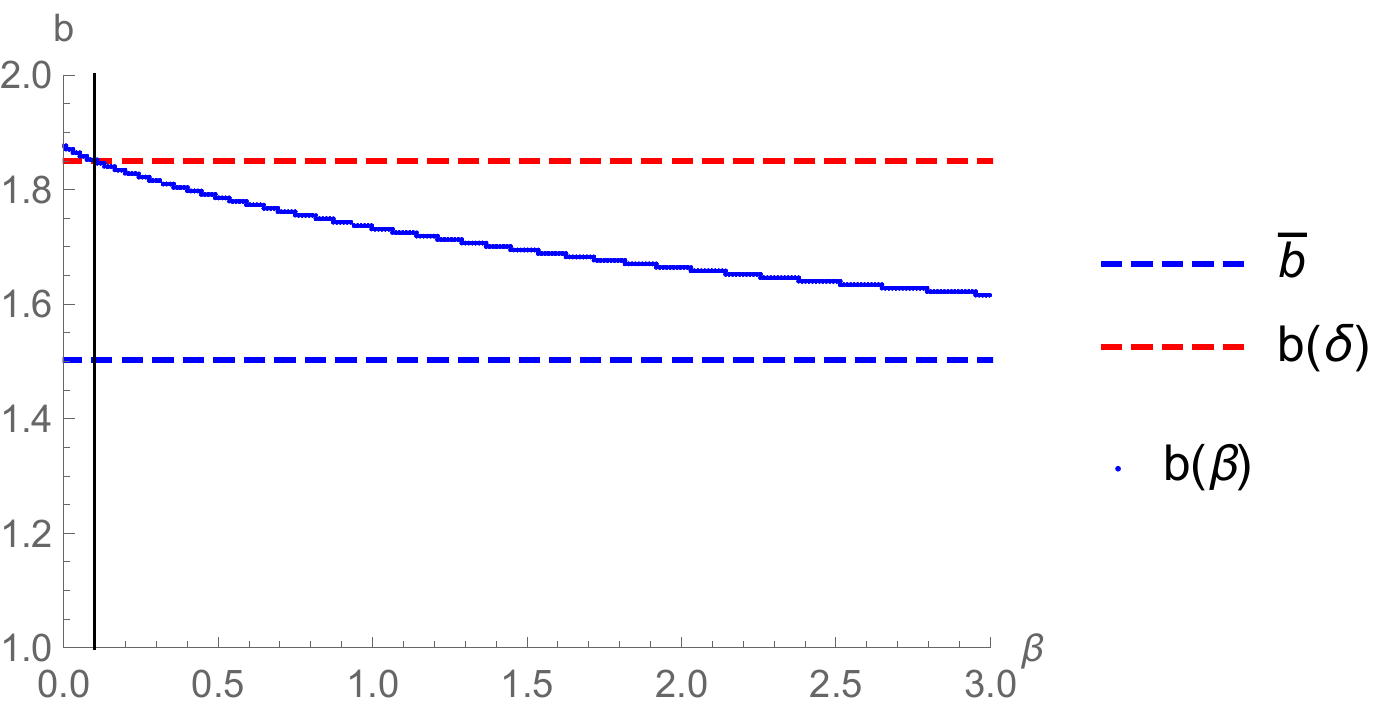}
\caption{We illustrate the dependence of the equilibrium threshold on $\beta$ in the case $\mu>L_{max}$.}\label{fig:barrier_in_beta_1}
\end{figure}

\begin{figure}[htb]\centering
\includegraphics[width=0.75\linewidth]{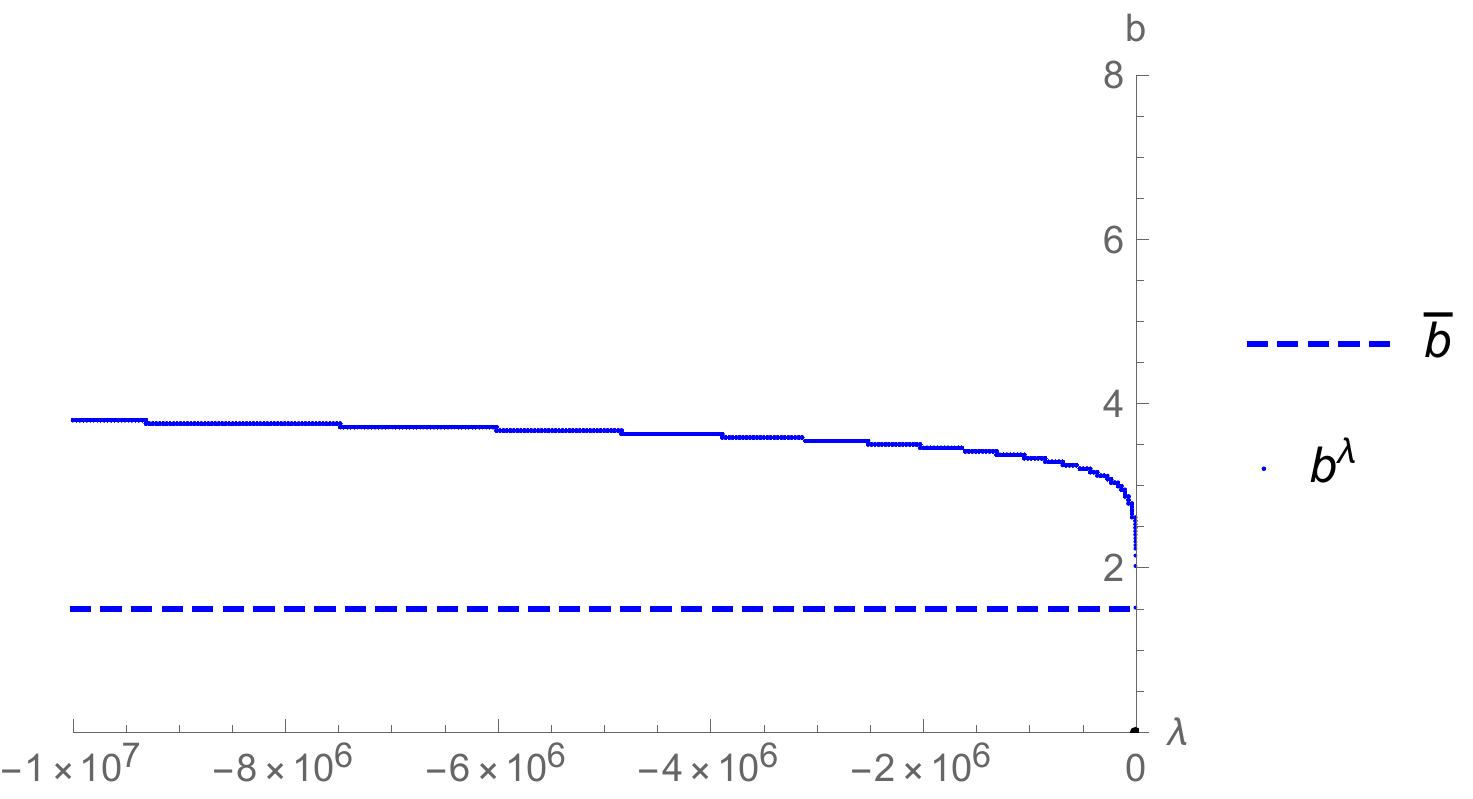}
\caption{The threshold dependent on $\lambda$ in the case $\mu>L_{max}$.}\label{fig:barrier_in_lambda_1}
\end{figure}

\subsection{The case \texorpdfstring{$\mu\leq L_{max}$}{TEXT}}

We consider again the parameters from Table \ref{tab:parameters} with the exception of $L_{max}=4$.
If we send $\beta$ to zero, we obtain that $b^*=\bar{b}$ since the equation $G_0(b)=0$ is independent of $\lambda$ as can be seen in Figure \ref{fig:barrier_in_beta_2}. Moreover, this yields that the solution $b^*$ equals the classical threshold $\bar{b}$ for all $\lambda\leq 0$. As in the above case, the equilibrium threshold tends to the classical optimal threshold level $\bar{b}$ for $\beta$ to infinity. Varying the value of $\lambda$ and keeping $\beta>0$ constant, we obtain a similar behaviour as in the previous case, which can be observed in Figure \ref{fig:barrier_in_lambda_2}. Furthermore, in Figure \ref{fig:barrier_in_lambda_2} (as in Figure \ref{fig:barrier_in_lambda_1} above) we notice that the equilibrium threshold increases very slowly if we decrease $\lambda$. In fact, the equilibrium threshold remains close to the classical threshold $\bar{b}$ even for very small values of $\lambda$, in comparison to the other parameters. This indicates that, even for relatively small $\lambda$, the penalty in case of ruin is not deterrent enough to cause a high threshold level.
\begin{figure}[htb]\centering
\includegraphics[width=0.75\linewidth]{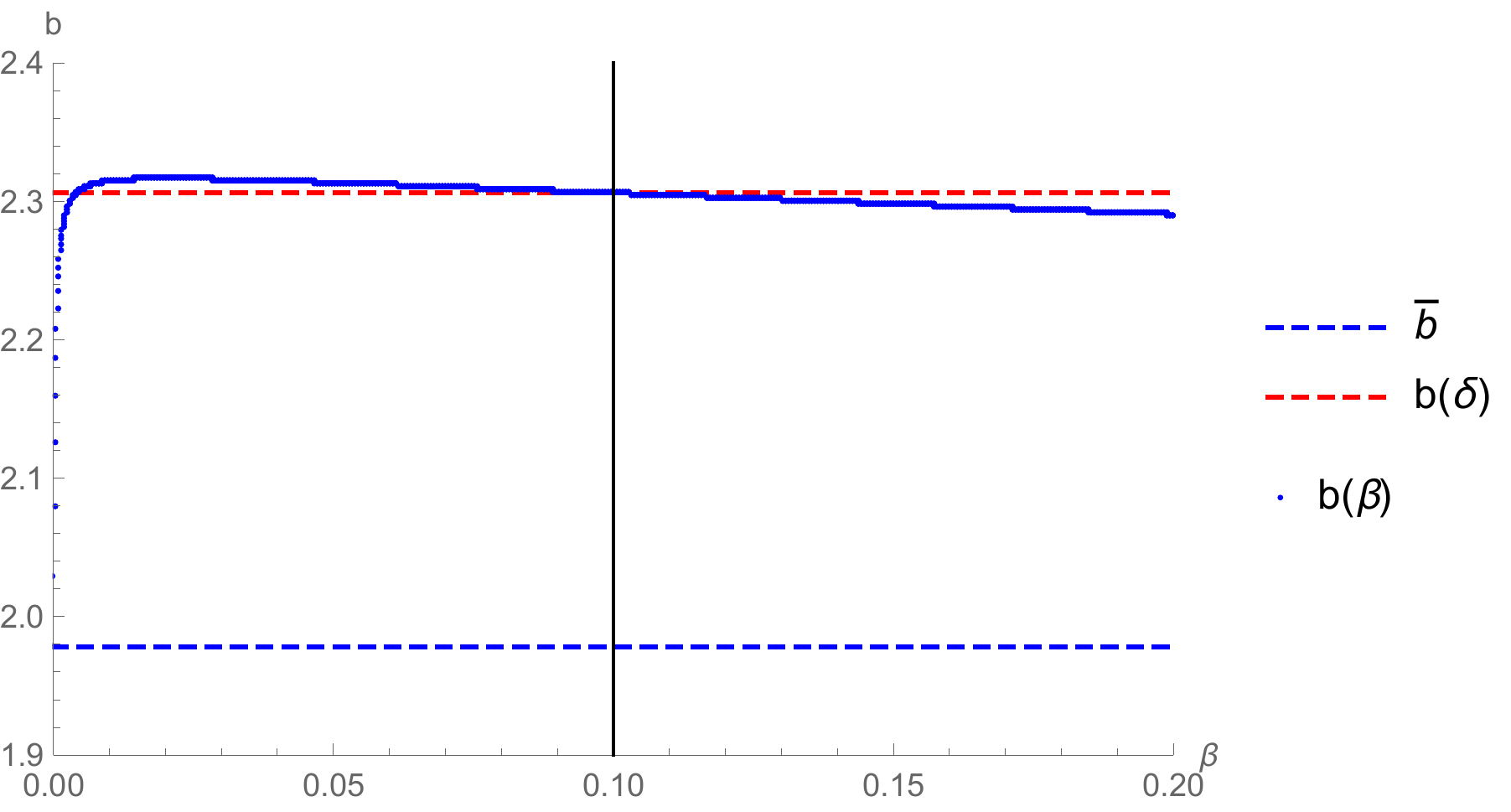}
\caption{Here we illustrate the dependence of the equilibrium threshold on $\beta$ in the case $\mu\leq L_{max}$.}\label{fig:barrier_in_beta_2}
\end{figure}
\begin{figure}[htb]\centering
\includegraphics[width=0.75\linewidth]{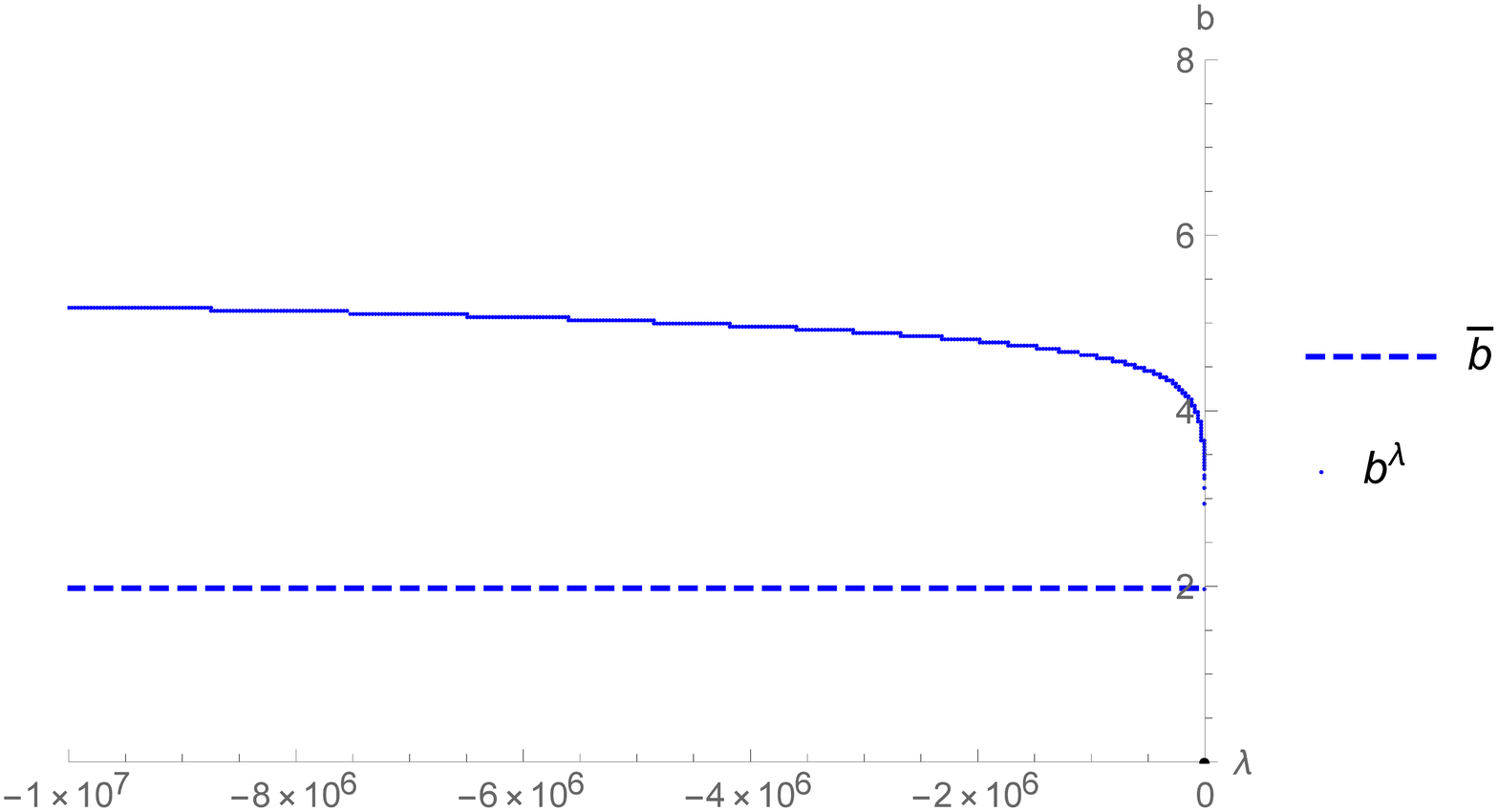}
\caption{The threshold dependent on $\lambda$ in the case $\mu\leq L_{max}$ and for $\beta=0.2$.}\label{fig:barrier_in_lambda_2}
\end{figure}

\section{Conclusion}
In this contribution we solve a time-inconsistent variation of the dividend maximization problem for a diffusion surplus process. This problem arises from imposing a constraint on the Laplace transform of the ruin time in the classical dividend problem. In order to include even a constraint on the ruin probability, we consider different discount rates in the dividend and the penalty term, which in fact leads to the time-inconsistency of the problem.\\
The presented verification theorem, which links an extended system of HJB-equations to an equilibrium dividend strategy and its value function, is applicable for state dependent diffusion coefficients. In the particular case of constant coefficients - the classical diffusion approximation - we are able to construct an explicit solution. This allows us to link the proposed reward function to the constrained dividend problem. In fact, we are able to fulfil the original constraint in a precommitment sense, i.e. for a given initial surplus level. Moreover, the reward function of the problem including a penalty coincides with the reward function of the classical dividend problem at this given initial surplus, since the penalty term vanishes.\\
Finally, the obtained equilibrium strategy is of threshold type, which is based on the restriction to Markovian controls and partly by the focus on an infinite time horizon.
This has the consequence that a constraint on the ruin probability itself can only be feasible if the drift term of the surplus process is larger than the maximal payout rate. Otherwise, the equilibrium is achieved by ignoring the constraint. In conclusion, the time-inconsistent approach allows us to assess the performance of a strategy in terms of an equilibrium over time.
%

\section{Funding Information}
This research was funded in whole, or in part, by the Austrian Science Fund (FWF) P 33317. For the purpose of open access, the author has applied a CC BY public copyright licence to any Author Accepted Manuscript version arising from this submission.

\bibliography{thesis_strini_bibliography}
\bibliographystyle{apalike}

\end{document}